\documentclass[11pt]{article}
\usepackage{amsfonts, epsfig, amsmath, amssymb, color}
\usepackage{textcomp}
\usepackage{paralist}

\usepackage[english]{babel}

\newcommand{\E}{\mathbf{E}}

\renewcommand {\epsilon}{\varepsilon}

\newcommand{\EE}{\mathbb{E}}

\newcommand{\NN}{\mathbb{N}}

\newcommand{\PP}{\mathbb{P}}

\newcommand{\RR}{\mathbb{R}}

\newcommand{\TT}{\mathbb{T}}

\newcommand{\aA}{\mathcal{A}}
\newcommand{\bB}{\mathcal{B}}
\newcommand{\cC}{\mathcal{C}}

\newcommand{\eE}{\mathcal{E}}
\newcommand{\fF}{\mathcal{F}}

\newcommand{\iI}{\mathcal{I}}

\newcommand{\oO}{\mathcal{O}}

\newcommand{\sS}{\mathcal{S}}
\newcommand{\tT}{\mathcal{T}}

\newcommand{\fB}{\mathfrak{B}}

\newcommand{\e}{\varepsilon}
\newcommand{\la}{\lambda}

\newcommand{\pd}{\partial}

\newcommand{\ra}{\rightarrow}

\newcommand{\ti}{\tilde}

\newcommand{\ds}{\displaystyle}
\newcommand{\ind}{\mathbf{1}}

\newcommand{\lqq}{\leqslant}
\newcommand{\gqq}{\geqslant}

\newtheorem{thm}{Theorem}[section]
\newtheorem{dfn}{Definition}[section]
\newtheorem{prp}{Proposition}[section]
\newtheorem{cor}{Corollary}[section]

\newtheorem{lem}{Lemma}[section]

\newtheorem{rem}{Remark}[section]

\DeclareMathSymbol{\ophi}{\mathalpha}{letters}{"1E}

\renewcommand{\phi}{\varphi}

\newcommand{\be}{\begin{equation}}
\newcommand{\ee}{\end{equation}}
\newcommand{\ben}{\begin{equation*}}
\newcommand{\een}{\end{equation*}}

\newcommand{\ba}{\begin{equation}\begin{aligned}}
\newcommand{\ea}{\end{aligned}\end{equation}}

\DeclareMathOperator{\trace}{trace}

\DeclareMathOperator{\supp}{supp}

\DeclareMathOperator{\dist}{dist}

\newenvironment{proof}{\par\noindent{\bf Proof:}}{\hfill$\blacksquare$\par}

\newcommand{\cA}{\mathcal{A}}

\newcommand{\bI}{\mathbf{1}} %\mathbb{J}
\newcommand{\bR}{\mathbb{R}}

\newfont{\cyrfnt}{wncyr10}
\def\J3{\cyrfnt{\rm \u{\cyrfnt I}}}
\def\j3{\cyrfnt{\rm \u{\cyrfnt i}}}

\usepackage[]{color}
\definecolor{DarkGreen}{rgb}{0.1,0.7,0.3}   %define a custom color

\definecolor{DarkGreen}{rgb}{0.1,0.7,0.3}   %define a custom color

\allowdisplaybreaks[4]

\usepackage[displaymath, mathlines, pagewise]{lineno}
\begin{document}
\title{The exit problem from the neighborhood of a global attractor 
for heavy-tailed L\'evy diffusions 
}

\date{\null}

\author{Michael H\"ogele\footnote{Institut f\"ur Mathematik, Universit\"at Potsdam,
Am Neuen Palais 10, 14465 Potsdam, Germany; hoegele@uni-potsdam.de}  \ 
and 
Ilya Pavlyukevich\footnote{Institut f\"ur Stochastik, Friedrich--Schiller--Universit\"at Jena, Ernst--Abbe--Platz 2, 
07743 Jena, Germany; ilya.pavlyukevich@uni-jena.de}
}

\maketitle

\begin{abstract}
We consider a finite dimensional deterministic dynamical system with a global attractor 
$\aA$ with a unique ergodic measure $P$ concentrated on it, 
which is uniformly parametrized by the mean of the trajectories in a bounded set $D$ containing $\aA$. 
We perturbe this dynamical system by a multiplicative heavy tailed L\'evy noise of small intensity $\e>0$
and solve the asymptotic first exit time and location problem 
from a bounded domain $D$ around the attractor $\aA$ in the limit of $\e\searrow 0$. 
In contrast to the case of Gaussian perturbations,
the exit time has the asymptotically algebraic exit rate as a function of $\e$, 
just as in the case when $\aA$ is a stable fixed point (see for instance \cite{DHI13,ImkellerP-06,Pavlyukevich11}).
In the small noise limit, we determine the joint law of the first time  and the exit location from $D^c$.  
As an example, we study the first exit problem from a neighbourhood of a stable limit cycle for the Van der Pol oscillator 
perturbed by multiplicative $\alpha$-stable L\'evy noise.

\end{abstract}



\section{Introduction}

This article studies perturbations of finite dimensional  dynamical systems by small multiplicative L\'evy noise with 
heavy-tailed large jumps with the focus on the exit behavior
from a bounded neighborhood of those global attractor. The scenario we shall study is as follows. 

Let us consider a $d$-dimensional deterministic dynamical system $\dot u=f(u)$ 
on a positively invariant 
bounded domain $D$. 
We assume that the dynamical system has a global attractor $\aA$ in $D$
and that uniformly over the initial conditions in $D$ the time averages of the trajectories 
converge to a unique invariant measure $P$ on $\cA$.
The most prominent examples of systems satisfying these settings are dynamical systems 
with a stable fixed point $\aA= \{\mathfrak{s}\}$ or a stable limit cycle $\aA = \oO$. 
Clearly, in this case the paths of the dymamical system never leave $D$. 

This situation changes significantly in the presence of a perturbation 
by noise, however small its intensity $\e>0$ may be. 
In the generic situation, the perturbed solution always exits from $D$. However the growth rate 
of the exit time shows an asymptotic behavior that strongly depends of the nature of the noise. 
Without any doubt, beginning with the pioneering works by Kramers \cite{Kramers-40} and Freidlin and Wentzell \cite{FW70}, 
the case of Gaussian perturbations has been studied quite exhaustingly 
in the realm of the large deviation theory. 
The literature on large deviation principles is enormous and representative examples 
for finite and infinite dimensional systems contain the works 
\cite{Br91,Br96,BovierEGK-04, Day-83, FL82, Fr88}. 
where perturbed gradient dynamical systems were mainly considered. 
For the case of non-gradient and degenerate systems we refer to \cite{BerglundG-04, Da96, FreidlinW-98}. 
They all have in common that the first exit time rate grows in $\e$ with the order $\exp(\bar V/\e^2)$, in physics literature known 
as Kramer's law, where is $\bar V$ the minimal amount of energy needed for a Brownian path to steer the perturbed system 
from the attractor $\aA$ to a point on the boundary $\partial D$. In other words, 
$\bar V$ depends only on the dynamical system outside the attractor. 
The dynamics on the attractor, where no energy is needed to travel, 
is irrelevant. 

The exit scenario changes fundamentally if the perturbation is a L\'evy process, 
with power tailed (heavy tailed) large jumps. 
In this case, the large jumps determine the exit behavior: It is possible to perform a time scales separation of big jumps 
from the small jumps and the Gaussian component such that on the new time scale the system's small noise 
behavior becomes essentially one 
of a deterministic system perturbed by large jumps. 
Using this approach,
the gradient case or the case with point attractors in finite and infinite dimensional systems 
has been treated in \cite{DHI13, Godovanchuk-82,ImkellerP-06,ImkPavSta-10,Pavlyukevich11}. 
Since the deterministic system converges to the stable state fast enough 
in comparison to the occurence rate of large jumps, 
the exit occurs when a system jumps from a vicinity of $\mathfrak{s}$. 
The resulting exit rate turns out to be of a power order with respect to $1/\e$, 
and  the asymptotic exit location in $D^c$ is given by the probability 
distribution of large jump increments conditioned to $D^c-\mathfrak{s}$. 
This is radically different from the case of Gaussian perturbations, where the exit occurs only on the boundary of $D$ 
due to the continuity of the paths.   

In the present paper, we generalize these results to the case where the global attractor $\aA$ in $D$ 
is not necessarily a stable point. Once again the essential exit behavior is determined 
by the deterministic system perturbed by large jumps. 
However we face the problem that --- opposite to the case of gradient systems --- 
the convergence of the deterministic trajectory 
to a hyperbolic attractor as a set does not imply the convergence 
towards a trajectory on the attractor. 
Instead, what replaces the deterministic 
control of the trajectory is its ergodic behavior, 
that is its ``occupation statistics'' of its time-average on the attractor.  
In this sense the exit event will be asymptotically triggered by the large jumps 
starting on $\aA$ under the invariant measure $P$. 
The exit rate is again of a power order in $1/\e$, but the precise prefactor 
depends now on the large jump distribution 
and the ergodic measure $P$ concentrated on $\aA$. 
The distribution of the exit location is hence given by the probability distribution 
for large jumps conditioned to $D-v$, where $v$ is averaged over $P$ on $\aA$. 
Therefore contrary to the aforementioned Gaussian case, 
the deterministic dynamics on the attractor turns out to be crucial 
for the asymptotics of the exit times.  

We can make this intuition rigorous for a very general class of additive and multiplicative 
L\'evy noises with a regularly varying L\'evy measure of index $-\alpha$, $\alpha >0$. 
In particular, our main result covers perturbations in It\^o and Stratonovich, as well as in the canonical (Marcus) 
integrals sense, where jumps in general do not occur along straight lines, 
but follow the flow of a vector field which determines the multiplicative noise.

Limit cycles attractors perturbed by Gaussian noise are considered in the physics and other natural sciences 
literature since quite some time \cite{EFSV85, GM88, HLP09, KurSch91, LC98, SV87}. As an application of our main result 
we work out the example of the Van der Pol oscillator perturbed by multiplicative $\alpha$-stable noise. 

It has been well-known for a long time, that the first exit time and location problem 
for general Markov processes can be stated in terms a Poisson and Dirichlet problem 
of the generator of this process, consult for instance \cite{Wentzell-90}.
However, the generators of the jump part in the case of L\'evy processes are non-local integro-differential operators, 
for which these problems are hard to solve, in particular in the case of the canonical Marcus noise. 
The advantage of our approach is among others the insensitivity to the boundary regularity of $D$ 
and the intuitive simplicity of the result.

\section{Object of study and main result}\label{sec: object of study}
\subsection{Deterministic dynamics}\label{subsec: deterministic dynamics}
Consider a bounded domain $D\subset \RR^d, d\gqq 1$ with piecewise $\cC^1$-smooth boundary 
and a vector field $f\in \cC^2(D, \RR^d)$, which points uniformly inward at the boundary. 
We are interested in the $d$-dimensional dynamical system given 
as the solution map $(t,x) \mapsto u(t; x),$ \mbox{$t\gqq 0,$} $x \in D$ 
of the autonomous ordinary differential equation 
\ba
\label{eq: det}
\dot u&=f(u),\quad\quad u(0) = x.
\ea
We further assume that the unique solution exists for all $x\in D$ and $t\geq 0$. 
Further we assume that the dynamical system defined by (\ref{eq: det}) has 
a global attractor $\aA$ in $D$. 

\begin{rem}\label{rem D.2}
Since by definition the global attractor attracts bounded sets in $D$, see for instance \cite{Temam97},
there exists a positively invariant set $\iI$ with $\aA \subset \iI \subset D$ 
such that $\dist(\pd D, \pd \iI)>0$ for which there is 
a time $\sS>0$ such that for all  $x\in D$ and $t\gqq \sS$
\begin{equation}
 u(t;x) \in \iI.
\end{equation}
\end{rem}

\paragraph{(D.1)} 
Let there exist a unique invariant probability measure $P$ on $\mathfrak{B}(\RR^d)$ with $\supp(P) = \aA$ 
such that all non-negative, measurable and bounded functions $\phi: \RR^d \ra \RR$ satisfy 
\begin{equation}
\lim_{t\ra \infty} \sup_{x\in D} \frac{1}{t} \int_0^t  \phi(u(s;x)) ds = \int_\aA \phi(v) P(dv).
\end{equation}

\begin{dfn} 
For $\delta>0$ we define the reduced domain of attraction 
\begin{align*}
D_{\delta} &:= D \setminus \bB_{\delta}(\partial D).
\end{align*}
\end{dfn}

\begin{rem} \label{lrm: properties of reduced domains} %Let Hypothesis \textbf{(D.1)} be fulfilled. 
Due to the assumption on the uniform inward pointing of $f$ at $\partial D$, 
there is $\delta_0\in (0,1)$ such that for all $\delta\in (0, \delta_0]$
\[
u(t, D_{\delta}) \subset D_{\delta} \qquad \mbox{ for all }t\gqq 0.
\]
\end{rem}

\subsection{The probabilistic perturbation }

\noindent 

On a filtered probability space $(\Omega, \fF ,\mathbb{P}, (\fF_t)_{t\gqq 0})$, satisfying 
the usual hypothesis in the sense of \cite{Protter-04}, 
we consider a L\'evy process $Z = (Z_t)_{t\gqq 0}$ with values in $\RR^m$, $m\gqq 1$ 
and the following characteristic function
\begin{equation}
 \E e^{i\langle u, Z_1 \rangle}=\exp\Big( -\frac{\langle Au,u\rangle}{2}  
+i\langle b,u\rangle +\int( e^{i\langle u, z \rangle}-1-i\langle u,z\rangle\bI (\|z\|\geq 1)\nu(dz)\Big),\ u\in\bR^m.
\end{equation}

\noindent Let us denote by $N(dt, dz)$ the associated Poisson random measure with the intensity measure $dt \otimes \nu(dz)$ 
and the compensated Poisson random measure 
$\ti N(dt, dz) = N(dt, dz) - dt \nu(dz)$. Consequently, by the L\'evy--It\^o theorem \cite{Applebaum-09} 
the L\'evy process~$Z$ given above has the following almost surely pathwise additive decomposition  
\begin{equation}\label{eq: Levy-Ito}
Z_t=b t+ A^{\frac{1}{2}} B_t+\int_{(0, t]}\int_{0<\|z\|<1} z \ti N(ds, dz) + \int_{(0,t]}\int_{\|z\|\gqq 1} z N(ds, dz),\qquad t\gqq 0,
\end{equation}
with $B = (B_t)_{t\gqq 0}$ a standard Brownian motion in $\RR^m$. 
Furthermore, the random summands in (\ref{eq: Levy-Ito}) are independent. 
For further details on L\'evy processes we refer to \cite{Applebaum-09} and \cite{Sato-99}.\\

\noindent \textbf{(S.1)} The L\'evy measure $\nu$ of the process $Z$ is 
\textbf{regularly \mbox{varying at $\infty$}} with index $-\alpha$. 
Let $h\colon (0,\infty)\to (0,\infty)$ denote its tail, 
\begin{equation}\label{def: h}
h(r):=\int_{\|y\|\geq r }\nu(dy).
\end{equation}
Then there exist $\alpha>0$ and a non-trivial self-similar Radon measure $\mu$ on $\bar \bR^m\backslash\{0\}$
such that $\mu (\bar\bR^m\backslash \bR^m)=0$ and
  for any $a>0$ and any Borel set $A$ bounded away from the origin, $0\notin \overline{A}$ with $\mu(\partial A)=0$, 
the following limit holds true:
\begin{equation}\label{eq: regular variation}
\mu(aA)=\lim_{r\ra \infty} \frac{\nu(raA)}{h(r)} = 
\frac{1}{a^\alpha} \lim_{r\ra \infty} \frac{\nu(rA)}{h(r)}= \frac{1}{a^\alpha} \mu(A).
\end{equation}
In particular, following \cite{BinghamGT-87} there exists a positive function $\ell$ slowly varying at infinity 
such that 
\[
h(r) = \frac{1}{r^{\alpha} \ell(r)},\qquad\mbox{ for all} \quad r>0.
\]
The selfsimilarity property of the limit measure $\mu$ implies that $\mu$ assigns no mass to
spheres centred at the origin of $\bR^m$ and has no atoms.
For more information on multivariate heavy tails and regular variation we refer
the reader to Hult and Lindskog \cite{HultL-06-1} and Resnick \cite{Resnick-04}.\\

\noindent \textbf{(S.2)}  
\noindent Consider continuous maps $G\in \cC(\mathbb{R}^d\times \mathbb{R}^m, \mathbb{R}^d)$ 
and $F, H: \RR^d \ra \RR^d$ and fix the notation
\[
a(x, y) := F(x) F(y)^T \qquad \mbox{ for }x, y \in \RR^d.
\]
We assume that there exists $L>0$ such that
$f$, $G$, $H$ and $F$ satisfy the following properties. 

\begin{enumerate}
\item \textbf{Local Lipschitz conditions: } For all $x, y \in  D$ 
\begin{multline*}
\|f(x) -f(y)\|^2  + \|a(x,x) - 2 a(x, y) + a(y,y)\| + \|H(x) - H(y)\|^2\\
+ \|F(x) - F(y)\|^2 +  \int_{\bB_1} \|G(x,w)-G(y,w)\|^2 \nu(dw) \lqq L^2 \|x-y\|^2.
\end{multline*}

\item \textbf{Local boundedness: } For all $x\in D$ 
\begin{align*}
\|f(x)\|^2  + \|a(x,x)\| + \|H(x)\|^2 + \|F(x)\|^2 +  \int_{\bB_1} \|G(x,w)\|^2 \nu(dw) \lqq L^2 (1+ \|x\|^2).
\end{align*}

\item \textbf{Large jump coefficient:} %For all $R > R_0$ large enough 
For all $x, y \in D$ and $w\in \RR^m$
\begin{align*}
\|G(x,w) - G(y, w)\| \lqq L e^{L (\|w\| \wedge L)} \|x-y\|. 
\end{align*}
% where $\diag(D) := \sup\{\|x-y\|, x,y\in D\}$

\item \textbf{Local bound for $G$ in small balls:} 
There exists $\delta'>0$ such that for $w\in \bB_{\delta'}(0)$ 
\begin{align*}
\sup_{v\in \bB_{\delta'}(\aA)} \|G(v, w)\| \lqq L.  
\end{align*}
\end{enumerate}

\label{S.4 G to 0 for small balls}

\begin{prp}
\noindent Let the assumptions \textbf{(D.1)} and \textbf{(S.1-2)} be fulfilled.  
Then for $\e\in (0,1)$ and $x\in D$ the stochastic differential equation 
\begin{align}
X_{t,x}^\e &= x + \int_0^t f(X^\e_{s,x})ds + \e \int_0^t H(X^\e_{s,x}) b \;ds + \e \int_0^t F(X^\e_{s, x}) d (A^{\frac{1}{2}}B_s) \nonumber\\
&\qquad  + \int_0^t \int_{\|z\| \lqq 1} G(X^\e_{s-,x}, \e z) \ti N(ds, dz) +  \int_0^t \int_{\|z\| > 1} G(X^\e_{s-,x}, \e z) N(ds, dz).  
\label{eq: sde}
\end{align}
has a unique local strong solution process $(X^\e_{t \wedge \TT,x})_{t\gqq 0}$ with c\`adl\`ag paths in $\RR^d$ 
and defines a strong Markov process with respect to $(\fF_t)_{t\gqq 0}$, 
where $\TT = \TT_x(\e)$ is the first exit time 
\begin{align*}
\mathbb{T}_x(\e)&:=\inf\{t\gqq 0\colon X^\e_{t,x}\notin D\}, \qquad \e>0, x\in D.
\end{align*}
\end{prp}
\noindent The proof can be found for instance stated as Theorem 6.23 in \cite{Applebaum-09} on page 367. 

\noindent The multiplicative perturbations in the sense of It\^o, Fisk--Stratonovich or Marcus are of a special interest for applications.   
Assume that $Z$ is a pure jump process with $A=0$, $b=0$. For a globally Lipschitz continuous function 
$\Phi\colon \bR^d\to\bR^{d\times m}$ 
consider the It\^o and canonical SDEs
\begin{align}
\label{eq:ito}
X_t&=x+\int_0^t f(X_s)dt+\e\int_0^t \Phi(X_{s-})dZ_s,\\
\label{eq:marcus}
X_t^\diamond &=x+\int_0^t f(X_s^\diamond)dt+\e\int_0^t \Phi(X_{s-}^\diamond)\diamond dZ_s.
\end{align}
Then the It\^o SDE \eqref{eq:ito} is obtained from \eqref{eq: sde} with
\[
G(x,z):=x-\Phi(x)z
\]
and the Marcus SDE \eqref{eq:marcus} with 
\[
G(x,z):=\phi^z(x),
\]
where $\phi^z(x) = y(1;x)$ is the solution of the ordinary differential equation 
\[
\dot y(s) = \Phi(y(s)) z, \qquad y(0) =x , \quad s\in [0,1].   
\]
If $L$ is the Lipschitz constant of the matrix function $\Phi$ then the 
Gronwall lemma implies that
\[
\|G(x,z)-G(y,z)\| \lqq L e^{L\|z\|} \|x-y\|\qquad \forall x, y \in D, z\in \RR^m.
\]

\subsection{The main result}\label{subsec: the main results}

\noindent For $x\in \bR^d$, $U\in \fB(\RR^d)$ with $x\notin U$ %$\rho^1$ defined in \textbf{(S.3)} 
we denote the set of increments $z\in \RR^m$ %\cred{\setminus \bB_{\rho^1}}$, 
which send $x$ into $U$ by
\begin{align}\label{def: event E}
E^{U}(x)& :=\{z\in \RR^m\colon x + G(x,z)\in U\}. 
\end{align}

\noindent We define the following measure assigning for $U \in \fB(\RR^d)$ 
\begin{align*}
Q(U) &:=\int_{\aA} \mu(E^{U}(y))~P(dy).
\end{align*}

\begin{rem}
Clearly for 
\begin{align*}
\la_\e &:= \int_{\aA} \nu\Big(\frac{E^{D^c}(y)}{\e}\Big) P(dy) \qquad \mbox{ and }\qquad h_\e := h\Big(\frac{1}{\e}\Big), \quad\e\in (0,1)
\end{align*}
equation (\ref{eq: regular variation}) implies  
\begin{align*}
\lim_{\e\ra 0+} \frac{\la_\e}{h_\e} = Q(D^c).  
\end{align*}
\end{rem}

\begin{thm}\label{thm: first exit times}
Let Hypotheses \textbf{(D.1)} and \textbf{(S.1-2)} be fulfilled and suppose that $Q(\partial D) = 0$ and $Q(D^c) >0$.
Then for any $\gamma \in (0,\frac{1}{5})$ %we define $\delta_\e := \e^\gamma$. 
any $\theta>0$ and $U\in \mathfrak{B}(\RR^d)$ 
such that $Q(\partial U) =0$ 
the first exit time $\TT_y(\e)$ satisfies \\
\begin{align*}
\lim_{\e\to 0} \sup_{y\in D_{\e^\gamma}}
\Big|\EE\left[ e^{-\theta Q(D^c) h_\e  \TT_y(\e)} %Q(\iI_R \cap \Omega^{c}) h_\e
\ind\{X^\e_{\TT_y(\e) ,y}\in U\}\right]-
\frac{1}{1+ \theta} \frac{Q(U \cap D^c)}{Q(D^c)} \Big|=0.\\
\end{align*}
\end{thm}

\begin{cor}\label{cor: first exit times}
Under the assumptions of Theorem \ref{thm: first exit times} follows 
\begin{align*}
 Q(D^c) h_\e     \TT_x(\e) &\stackrel{d}{\to} \mbox{EXP}(1), \\[2mm]
\PP(X_{\TT_x(\e), x}^\e \in U) &\to \frac{Q(U\cap D)}{Q(D)},\quad \e\to 0,
\end{align*}
where the convergence is uniform over all initial values $x \in D_{\e^\gamma}$.   
\end{cor}

\subsection{Example: Van der Pol oscillator perturbed by $\alpha$-stable L\'evy noise }

As a simple but illustrative application of Theorem \ref{thm: first exit times} we determine the law  of the first exit
time of a Van der Pol oscillator perturbed by small It\^o-multiplicative $\alpha$-stable L\'evy noise.
More precisely, let $Z$ be a bivariate L\'evy process with the characteristic function
\[
\EE\left[ e^{i\langle u , Z_t \rangle}\right]=e^{-t c(\alpha) \|u\|^\alpha},\quad \alpha\in (0,2),\ u\in\bR^2,\quad 
c(\alpha)=\frac{\pi}{2^\alpha}\frac{|\Gamma(-\frac{\alpha}{2})|}{\Gamma(1+\frac{\alpha}{2})},
\]
and a L\'evy triplet $(0,\nu,0)$, where
\begin{align*}
\nu(dy)=\bI_{\bR^2\backslash \{0\}}(y)\frac{dy}{\|y\|^{2+\alpha}}.
\end{align*}
Clearly, $\nu$ is a regularly varying measure of index $-\alpha$ with the limit measure $\mu=\nu$ and a scaling function
\[
h(r)=\int_{\|y\|\geq r}\frac{dy}{\|y\|^{2+\alpha}}=\frac{2\pi}{\alpha}\frac{1}{r^\alpha}.
\]
Consider a  Van der Pol oscillator for $u = (u_1, u_2)$ and $f = (f_1, f_2)$ 
\begin{align*}
\dot u&=f(u),
\qquad 
\begin{cases} 
f_1(u_1, u_2)&=u_1,\\
f_2(u_1, u_2)&=-u_1 + (1-u_1^2)u_2.
\end{cases}
\end{align*}
which has an unstable stationary solution $u\equiv 0$ and a unique periodic solution
$u^\circ=(u_1^\circ(t),u_2^\circ(t))_{t\in [0, T^\circ ]}$ of period $T^\circ>0$ irrespective of initial values 
which we can omit since all quantities involved will not depend on them. 
It is well known that the set $\cA =\{ (u_1^\circ(t),u_2^\circ(t))_{t\in [0, T^\circ ]}\}\subset \bR^2$
is an exponentially orbitally stable limit cycle. 
In particular for any bounded and measurable function $\phi: \RR^d \ra (0,\infty)>0$, 
any initial point $x\neq 0$, we have 
\begin{align*}
&\frac{1}{t}\int_0^t \phi(u(s,x))\, ds\to \frac{1}{T^\circ} \int_0^{T^\circ} \phi(u^\circ(s)) ds =\int_\cA\phi(v)P(dv),\\
&\mbox{ where }\quad P(B)=\frac{1}{T^\circ} \int_0^{T^\circ} \bI_B(u^\circ(s)) ds,\quad \mbox{ for }B\in \fB(\RR^2),  
\end{align*}
and this convergence is uniform over all $x \in D$ bounded away from the origin.
Consider now a Van der Pol oscillator perturbed by multiplicative It\^o noise
\begin{align*}
dX^\e_t=f(X^\e_t)dt+\e G(X^\e_t) dZ_t
\end{align*}
where $(x_1, x_2) \mapsto G(x_1,x_2)$ is a $2\times 2$ matrix valued function satisfying Hypotheses (S.1) and (S.2) of Section 2. 
Let $D$ be an open bounded invariant domain of attraction containing the 
limit cycle $\cA$ with $\dist(\cA,\partial D)>0$, see Fig.~\ref{f:vdp}.
% In particular let $D^\times := D\setminus \{0\}$.
 \begin{figure}[t]
\begin{center} \hfill a)
 \includegraphics[width=7cm]{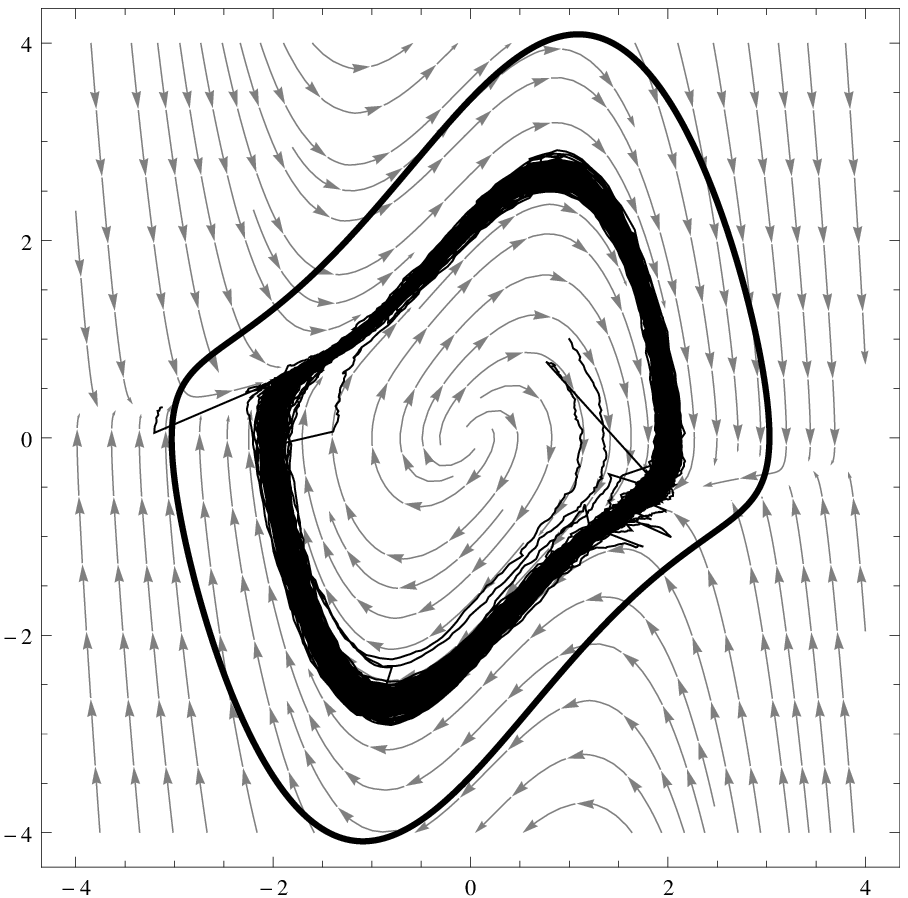}\hfill 
b)\includegraphics[width=7cm]{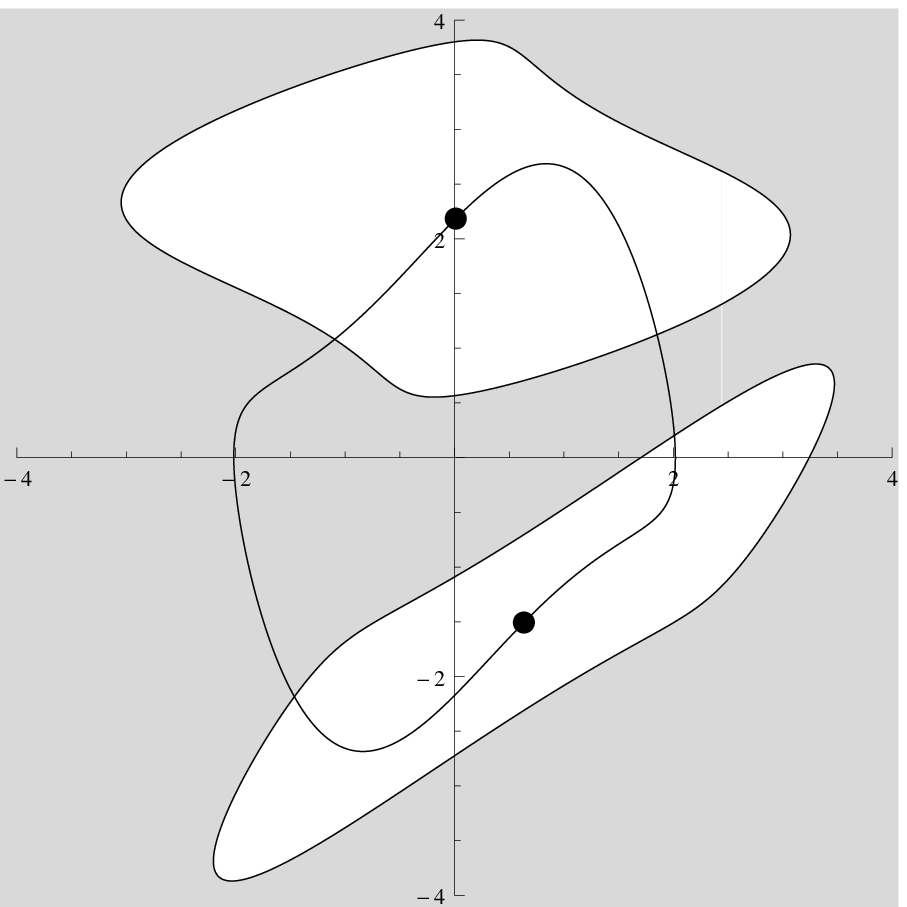}\hfill 
\end{center}
\caption{a) A typical exit path of a Van der Pol oscillator perturbed by 1.9-stable L\'evy noise; b) the domains 
$G^\ominus (D^c-u^\circ(t))$ in the space of noise jumps for two different values of $t\in[0, T^\circ]$.\label{f:vdp}}
\end{figure}
Let
\begin{align*}
G_t&= G(u^\circ_1(t),u^\circ_2(t))\quad  \text{ and }\quad  G^\ominus_t:=
\begin{cases}
G^{-1}_t, &\det G_t\neq 0,\\
0,&\text{ otherwise}.
\end{cases}
\end{align*}
For any $\delta>0$ we can choose a small neighbourhood $\bB_\delta(0)$ 
of the unstable fixed point~$0$ of the Van der Pol oscillator, such that
the domain $D^{(\delta)}=D \backslash \bB_\delta(0)$ and $f$ 
satisfy Hypothesis (D.1). 
Let $x\in D^{(\delta)}$. Denote by $\mathbb T^{(\delta)}_x(\e)$ the first exit time from the domain $D^{(\delta)}$. 
We are now in the state to apply Theorem~\ref{thm: first exit times} and find that
\begin{align*}
\e^\alpha \frac{2\pi}{\alpha T^\circ}
\int_0^{T^\circ}\Big[\int_{G_s^\ominus D^c} +\int_{G_s^\ominus \bB_{(\delta)}(0)}\frac{dy}{\|y-u^\circ(s)\|^{2+\alpha}}\, \Big]d s \cdot 
\mathbb T_x^{(\delta)}(\e) \to {EXP}(1),\quad \e\to 0.
\end{align*}
Taking into account that $\int_{\bB_{(\delta)}}  dy \to 0$ as $\delta\to 0$ we finally obtain 
the limiting law for $\mathbb T_y(\e)$ such that 
\begin{align*}
\e^\alpha 
\bigg(\frac{2\pi}{\alpha T^\circ}
\int_0^{T^\circ}\int_{G_s^\ominus D^c} \frac{dy}{\|y-u^\circ(s)\|^{2+\alpha}}\, d s\bigg) \cdot \mathbb T_x^{(0)}(\e)
\stackrel{d}{\to} EXP(1),\quad \e\to 0, 
\end{align*}
with the convergence uniformly over all $x\in D \setminus \bB_{\e^{\frac{1}{10}}}(\partial D^\times)$ with $D^\times = D\setminus \{0\}$.

\section{Small jumps dynamics\label{sec: small noise asymptotics}}

The aim of this section is to determine the precise asymptotics of 
$(X^\e_{t,x})_{t\in [0, T_1]}$ for the first large jump times $T_1$. 
This will be accomplished in Proposition \ref{prop: ergodicity}, 
which tell us that for times $t\in [0, T_1)$ the deterministic dynamics 
and its ergodicity property dominates, and at $t = T_1$ 
there occurs a single large jump. 
We assume Hypotheses \textbf{(D.1)} and \textbf{(S.1-2)} to be satisfied in the sequel.

\subsection{Asymptotics until the first large jump}%\label{sec: small noise asymptotics}

Let $\rho=\rho_\e$, $\e\in(0,1]$, be a positive sequence, which is monotonically increasing to infinity, 
$\rho_\e\nearrow \infty$ as $\e\searrow 0$ and denote by 
\begin{align*}
 \beta_\e := \nu(\bB_{\rho^\e}^c).
\end{align*}
Consider the following $\e$-dependent L\'evy-It\^o decomposition $Z_t :=  \xi^\e_t + \eta_t^\e$ for all $t \gqq 0$, $\e\in (0,1)$, 
\begin{equation}
\label{def: b eps}
\begin{aligned} 
&\eta^\e_t := \int_{(0,t]}\int_{\|z\| > \rho^\e } z N(ds, dz),\\
&\xi^\e_t := Z_t - \eta_t^\e = b_\e t  + A^{\frac{1}{2}} B_t+\int_{(0, t]}\int_{0<\|z\|\lqq \rho^\e} z \ti N(ds, dz),\\
&b_\e  := b + \EE[\int_{(0,1]}\int_{\{1< \|y\|\lqq \rho^\e\}} y N(ds, dz)] = b + \int_{1 < \|z\|\lqq \rho^\e} z \nu(dz).
\end{aligned}
\end{equation}
The compound Poisson process $\eta^\e$ here is characterised by a family of  i.i.d.\ waiting times $(\tau^\e_i)_{i\in \NN}$ 
with \mbox{$\tau_i^\e \sim \mbox{EXP}(\beta_\e)$,} the renewal times 
\[
T_i^\e = \sum_{k=1}^i \tau_k^\e, 
\]
and an family of  i.i.d.\ large jumps $(W^\e_i)_{i\in\NN}$, also independent of $(\tau^\e_i)_{i\in \NN}$ 
with $W^\e_i \sim \nu_\e$, where
\begin{equation}\label{eq: truncated nu}
\nu_\e(\cdot) = \frac{\nu(\cdot \cap \bB_{\rho^\e}^c)}{\nu(\bB_{\rho^\e}^c)}.
\end{equation}
The process $\xi^\e$ is a L\'evy process with jumps bounded from above by $\rho^\e$ and hence has all finite moments.

\subsection{Control of the small jump noise} 

In this subsection we show that the probabilities of deviations of 
bounded integrals driven by the small noise $\e \xi^\e$ defined in (S.2) 
decay exponentially. 

\begin{lem}
\label{lem: drift estimate} 
Let $(\delta_\e)_{\e\in(0,1]}$ be a monotone sequence with $\delta_\e\searrow 0$ as $\e\searrow 0$ satisfying in addition
\begin{equation}\label{eq: rho-delta limit 0}
\lim_{\e\ra 0+} \e \frac{\rho^\e}{\delta_\e} = 0. 
\end{equation}
Then for any $C>0$ there is $\e_0\in (0,1)$ such that for all $\e \in (0, \e_0]$ 
\begin{equation}
\label{eq: drift estimate}
\frac{\e \|b_\e\|}{\delta_\e} \lqq C. 
\end{equation}
\end{lem}

\begin{proof}
In order to prove (\ref{eq: drift estimate}) we center the process $\xi$, $\ti \xi_t := \xi_t - b_\e t$, such that $\ti \xi_t$ is a L\'evy martingale with jumps bounded 
from above by $\rho^\e$. Since $\|b_\e\| \lqq \|b\| + \|\int_{1< \|z\|\lqq \rho^\e} y \nu(dy)\|$ we obtain by Jensen's inequality 
and the regular variation of the function $h$ defined by (\ref{def: h}) that 
\begin{align*}
\Big\|\int_{1< \|z\|\lqq \rho^\e} y \nu(dy)\Big\|^2 &\lqq \int_{1 < \|z\|\lqq \rho^\e} \|y\|^2 \nu(dy) 
= -\int_1^{\rho^\e} r^2 h(dr) \lqq (\rho^\e)^2 h(1), 
\end{align*}
such that $\|b_\e\| \lqq \|b\| + \sqrt{h(1)} \rho^\e$, which gives the desired result with the help of (\ref{eq: rho-delta limit 0}).  
\end{proof}

\begin{lem}\label{lem: quadratic variation} 
Let $(\delta_\e)_{\e\in(0,1]}$ be a monotone sequence with $\delta_\e\searrow 0$ as $\e\searrow 0$
and $p\gqq 1$ satisfying  
\begin{equation}\label{eq: rho-delta limit}
\lim_{\e\ra 0+} \frac{\e\rho^\e}{\delta_{\e}^{(p+1)/2}} = 0. 
\end{equation}
Then for all $T>0$ and $C>0$ there is $\e_0\in (0,1)$ such that for all $\e \in (0, \e_0]$ 
\begin{align*}
\PP([\e  \xi]_{\tau^\e} > C \delta_\e^p) &\lqq e^{-C \delta_\e^{-1} +1}.
\end{align*}
\end{lem}

\begin{proof}
The discontinuous part of the quadratic variation process $[\e \ti  \xi]_t^d = [\e \ti  \xi]_t - \trace(A) \e^2 t $ 
is a L\'evy subordinator and has the representation 
\begin{align*}
[\e \ti \xi]_t^d 
= \e^2 \sum_{s\lqq t} \|\Delta \ti \xi\|^2_s 
= \e^2 \int_{(0, t]}\int_{0<\|z\|\lqq \rho^\e} \|z\|^2 N(dz, ds)\quad t\gqq 0\mbox{ a.s.}  
\end{align*}
Since the jumps of $[\e \ti  \xi]_t^d$ by construction are bounded by $(\e\rho^\e)^2\lqq 1$, 
its Laplace transform is well-defined for all $\la\in \RR$ and $t\gqq 0$
\begin{align*}
\EE\left[e^{\la [\e \ti \xi ]_t^d]}\right] 
&= \exp\big(t \int_{0< \|y\|\lqq \rho^\e} (e^{\la^2\e^2 \|y\|^2} -1) \nu(dy)\big)\\
&= \exp\big( - t\int_{0< r\lqq \rho^\e} (e^{\la^2\e^2 r^2}-1) h(dr) \big).
\end{align*}
For any $\la>0$ the exponential Chebyshev inequality yields 
\begin{align*}
 \PP\big([\e \ti  \xi]_{T}^d >C\delta_\e^p \big) &\lqq \PP\big(e^{\la [\e \ti  \xi]^d_{T}} > e^{\la \delta_\e}\big) 
 \lqq e^{-\la C \delta_\e^p} \EE\big[e^{\la [\e \ti  \xi]^d_{T}}\big] \\
&=  \exp\big(-\la C \delta_\e^p - T \int_{0< r\lqq \rho^\e} (e^{\la^2 \e^2 r^2} -1) h(dr) \big).
\end{align*}
We continue with the help of $e^{s}-1 \lqq 2s$ for small $s$. Replacing $\la$ by $\delta_\e^{-(p+1)}$
we ensure the smallness of the argument noting that by (\ref{eq: rho-delta limit}) 
$\sup_{0 < r\lqq \rho^\e} \e^2 r^2 /\delta_\e^{p+1} \lqq (\e \rho^\e)^2/\delta_\e^{p+1} \ra 0$ for $\e\ra 0+$. 
We obtain 
\begin{align*}
& \big|T \int_{0<r\lqq \rho^\e} (e^{\e^2 r^2/\delta_\e^{p+1}}-1) h(dr) \big| \\
& \lqq \big|2 T \e^2 /\delta_\e^{p+1} \big(\int_{0< r\lqq 1} + \int_{1 < r\lqq \rho^\e} \big) r^2 h(dr) \big|\\[1mm]
& \lqq 2 T \e^2 /\delta_\e^{p+1} \big|\int_{0< r\lqq 1} r^2 h(dr) \big| 
+ 2 T (\e \rho^\e)^2/\delta_\e^{p+1}   |\int_{1 < r\lqq \rho^\e} h(dr)\big|\\[2mm]
& \lqq c T (\e \rho^\e)^2/\delta_\e^{p+1}.
\end{align*}
Therefore by (\ref{eq: rho-delta limit}) there is $\e_0\in (0,1)$ such that $\e\in (0, \e_0]$ implies the final result
\begin{align*}
\PP([\e\ti  \xi]_{T} > C\delta_\e^p) & \lqq \exp\big(-C \delta_\e^{-1} + \trace(A) \e^{2} T 
+ cT (\e \rho^\e)^2/\delta_\e^{p+1}\big) 
\lqq \exp\big(-C\delta_\e^{-1}+1 \big).
\end{align*}
\end{proof}
\noindent In the following lemma we estimate the deviation of the stochastic integral with respect to the (local) martingale part $\ti \xi^\e$ 
of the small jumps noise process $\xi^\e$
\[
\ti \xi^\e_t = A^{1/2} B_t + \int_{0< \|y\| \lqq \rho^\e} y \ti N(t, dy).
\]

\begin{lem}\label{lem: bounded stoch int} Let $(g_t)_{t\gqq 0}$ be an adapted, c\`adl\`ag process with bounded values 
by $C_g$ in $\RR^{m\otimes d}$ for a suitable matrix norm. 
For all $T>0$ and functions $\delta_\e$ and $\rho^\e$ satisfying (\ref{eq: rho-delta limit}) for $p=4$
there is $\e_0\in (0,1)$ and a constant $C_0>0$ such that for $\e \in (0, \e_0]$ 
\begin{align*}
\PP(\sup_{s\in [0, T]} \e \sum_{i=1}^d \big|\sum_{j=1}^m \int_0^t g_{s-}^{ij} d\ti  \xi^j(s) \big| > \delta_\e) 
\lqq \exp(-C_0 \delta_\e^{-1} + \ln(6d)).
\end{align*}
\end{lem}

\begin{proof} Suppose $\max_{i,j} \sup_{t\gqq 0} |g_t^{ij}| \lqq C_g$ almost surely. 
We consider the each component of the $d$-dimensional martingale 
\begin{align*}
M_t^i = \sum_{j=1}^m \int_0^t g^{ij}_{s-} d\ti  \xi^j(s).
\end{align*}
By construction $\|\Delta_t M\| \lqq m d C_g \rho^\e =: C \rho^\e$ almost surely. 
We estimate the probability of a deviation of size $\delta_\e$ from zero 
conditioned on small quadratic variation
\begin{equation}\label{eq: estimate of small noise integral} 
\PP(\sup_{s\in [0,T]} \|\e M_s\| > \delta_\e) 
\lqq \PP(\sup_{s\in [0,T]} \|\e M_s\| > \delta_\e ~|~ [\e M]_{T} \lqq \delta_\e^4) 
+ \PP([\e M]_{T} > \delta_\e^4).
\end{equation}
\textbf{Step 1:} We estimate the first term of inequality (\ref{eq: estimate of small noise integral}). 
Following the lines of the proofs of Lemma~{26.19} and Theorem~{26.17} part (i)
 in \cite{Kallenberg-02} 
we find the following estimate. For any $\la>0$  
\begin{equation*}
\PP(\sup_{s\in [0,T]} \e M_s^i > \delta_\e ~|~ [\e M]_{T} \lqq \delta_\e^4) \lqq 
\exp\big(-\la \delta_\e + \la^2 \Upsilon(\la C_g \e \rho^\e) \delta_\e^4 \big), 
\end{equation*}
where $\Upsilon: (0, \infty) \ra (0, \infty), \Upsilon(x) = - (x + \ln(1-x)_+) x^{-2}$. 
Replacing $\la$ by $\la_\e = \delta_\e^{-2}$ and keeping in mind that 
$\lim_{\e \ra 0+} \Upsilon(\la C_g  \e \rho^\e) = \frac{1}{2}$ 
yields 
\begin{equation*}
\PP(\sup_{s\in [0,T]} \e M_s^i > \delta_\e ~|~ [\e M]_{T} \lqq \delta_\e^4) \lqq 
\exp\big(-\delta_\e^{-1}). 
\end{equation*}
For the infimum of the negative analogue holds the respective estimate, 
which provides for each $i$ for $\la_\e = d\delta_{\e}^{-2}$ instead
\begin{equation*}
\PP(\sup_{s\in [0,T]} |\e M_s^i| > \frac{\delta_\e}{d} ~|~ [\e M]_{T} \lqq \delta_\e^4) \lqq \exp\big(-\delta_\e^{-1}+\ln(2)), 
 \end{equation*}
where the right-hand side does not depend on $i$, such that eventually
\begin{align*}
&\PP(\sup_{s\in [0,T]} \|\e M_s\| > \delta_\e ~|~ [\e M]_{T} \lqq \delta_\e^4) \\
& \lqq \sum_{i=1}^d \PP(\sup_{s\in [0,T]} |\e M_s^i| > \frac{\delta_\e}{m} ~|~ [\e M]_{T} \lqq \delta_\e^4)\\[2mm]
&\lqq \exp\big(-\delta_\e^{-1}+\ln(2d)). 
\end{align*}
\textbf{Step 2:} We treat the second term in inequality (\ref{eq: estimate of small noise integral}). 
The boundedness assumption of $g$ yields 
\begin{align*}
[\e M]_t = \int_0^t \|g_{s-}^* g_{s-}\|^2 d[\e A^{\frac{1}{2}}B]_s + \int_0^t \|g_{s-}^* g_{s-}\|^2 d[\e  \ti \xi]_s^d 
\lqq C^2 (\e^2 \trace(A) t+ [\e \ti  \xi]_t^d), \quad t\gqq 0. 
\end{align*}
Hence 
\begin{align*}
\PP([\e M]_{T} \gqq \delta_\e^4) \lqq \PP(C^2 [\e \ti  \xi]_{T}^d
\gqq \frac{1}{2} \delta_\e^4) + \PP( C^2 \trace(A) \e^{2} T \gqq \frac{1}{2}\delta_\e^4).
\end{align*}
The second term vanishes by (\ref{eq: rho-delta limit}), 
which implies $\e^2  < \delta_\e^4$ for small $\e\in (0,1)$. 
The first term is treated as in Lemma~\ref{lem: quadratic variation}. 
Eventually 
\begin{align*}
\PP([\e M]_{T} \gqq \delta_\e^4) 
\lqq \PP([\e \ti  \xi]_{T}^d \gqq \frac{1}{2 C^2}\delta_\e^4)
\lqq  \exp(-\frac{\delta_\e^{-1}}{2 C^2} + 1).  
\end{align*}
Combining Step 1 and 2 yields a constant $\e_0\in (0,1)$ such that for all $\e\in (0, \e_0]$ 
\begin{align*}
\PP(\sup_{s\in [0,T]} \|\e M_s\| > \delta_\e) \lqq  \exp(-\min(1, \frac{1}{2 C^2})\delta_\e^{-1} + \ln(2de)). 
\end{align*}
This finishes the proof. 
\end{proof}

\subsection{Localization of $V^\e$ close to $u$ up to a fixed time}

Let $V^\e$ be the solution of equation (\ref{eq: sde}), 
where the driving noise $Z$ is replaced by the $\e$-dependent 
small jumps part $\xi^\e$ of $Z$ as definied in (\ref{def: b eps}). 
The first large jump time $T_1>0$ is exponentially distributed by with intensity $\beta_\e \searrow 0$ 
as $\e \searrow 0$. By definition then 
\[
V^\e_{t, x} = X^\e_{t, x} \qquad \mbox{ for } t\in [0, T_1). 
\]
In order to study the fluctions of $X^\e_{t,x}$ for $t< T_1$ we introduce 
\[
\TT^*_x(\e) := \inf\{t>0~|~V^\e_{t,x} \notin D\}.
\]

\begin{lem}[Non-exit up to fixed times]\label{lem: localization in a ball}
For any $T\gqq 0$ there is $\e_0\in (0,1)$ such that for all $\e \in (0, \e_0]$ 
and $\delta_\e$ satisfying (\ref{eq: rho-delta limit}) there
\begin{align*}
&~  \sup_{x\in D_{\delta_\e}} \PP(\TT^*_x(\e) \lqq T) \lqq \exp(-\delta_\e^{-1} + 2+ \ln(d)).
\end{align*}
\end{lem}

\begin{proof} 
By Remark \ref{lrm: properties of reduced domains} for any sufficiently small $\delta_\e$ and $x\in D_{\delta_\e}$ follows 
\[
\dist(u(t;x), \pd D)\gqq \delta_\e \qquad \forall t\gqq 0.  
\]
Since $\TT_x^*(\e)$ denotes the exit from $D$, we infer  
\begin{align*}
\{\TT_x^*(\e) \lqq T\} &= \{\TT_x^*(\e) \lqq T\}\cap \{\sup_{t\in [0, \TT_x^*(\e)]} \|V_{t,x} - u(t;x)\| > \delta_\e\}. 
\end{align*}
We lighten notatoin $V = V^\e$, $\TT^* = \TT^+_x(\e)$ etc. Then for $t\lqq T$ follows by definition 
\begin{align}
& V_{t\wedge \TT^*,x} - u(t\wedge \TT^*; x)\nonumber\\
& = \int_{0}^{t\wedge \TT^*} f(V_{s,x}) - f(u(s; x)) ds +  
\e \int_{0}^{t\wedge \TT^*}  H(V_{s,x}) b_\e ds + \e \int_{0}^{t\wedge \TT^*} F(V_{s,x}) dA^{\frac{1}{2}} B_s  \nonumber\\
& \qquad + \int_{0}^{t\wedge \TT^*}\int_{0< \|z\|\lqq \rho^\e} G(V_{s-,x}, \e z) \ti N(ds, dz). \label{eq: V-u}
\end{align}
We fix the constant 
\begin{equation}\label{def: CR}
C_{D} := \sup_{\substack{v\in D\\ w\in \bB_1}}\max\{L, \|f(v)\|, \|H(v)\|, \|F(v)\|, \|G(v,w)\|\}. 
\end{equation}
The global Lipschitz property of $f$ on $D$ and 
the standard integral version of Gronwall's lemma 
yield
\begin{multline}\label{eq: sup V-u}
\sup_{x\in D_{\delta_\e}} \sup_{t\in [0, T\wedge \TT_x^*]} \|V_{t,x} - u(t; x)\| \\
\lqq e^{C_{D} T} \sup_{x\in D_{\delta_\e}}\sup_{t\in [0, T\wedge\TT_x]}
 \|\e \int_{0}^{t}  H(V_{s,x}) b_\e ds + \e \int_{0}^{t} F(V_{s,x}) dA^{\frac{1}{2}} B_s \\
 + \int_{0}^{t}\int_{0< \|z\|\lqq \rho^\e} G(V_{s-,x}, \e z) \ti N(ds, dz)\|.
\end{multline}
The representation (\ref{eq: V-u}) 
has the (local) martingale part 
\begin{equation}
M_{t,x} := \e \int_0^t F(V_{s, x}) d (A^{\frac{1}{2}}B_s) 
+ \int_0^t \int_{0<\|z\| \lqq \rho^\e} G(V_{s-,x}, \e z) \ti N(ds, dz). \label{eq: def m}
\end{equation}
The previous lemma yields for $i$-th component $M^i_{t,x}$ 
and any $\la>0$ 
\begin{align}
& \sup_{x\in D_{\delta_\e}} \PP(\TT_x^*(\e) \lqq T) \nonumber\\
& = \sup_{x\in D_{\delta_\e}} \PP( \TT_x^*(\e) \lqq T, 
\sup_{t\in [0, T \wedge \TT^*]} \|V_{t,x}- u(t;x)\| > \delta_\e) \nonumber\\ 
& \lqq \sup_{x\in D_{\delta_\e}} \PP(\TT^*_x(\e) \lqq T, 
\sup_{t\in [0, T \wedge \TT^*] } e^{C_{D} T} \e \|\int_0^t H(V_{s,x}) b_\e ds\| > \frac{\delta_\e}{2}) \nonumber\\
& \qquad + \sup_{x\in D_{\delta_\e}} \PP( \TT^*_x(\e) \lqq T, 
\sup_{t\in [0,T \wedge \TT^*]} e^{C_{D} T} \|M_{t,x}\| > \frac{\delta_\e}{2}~|~[\e \xi]_{T} \lqq \delta_\e^4) \nonumber\\
& \qquad + \PP([\e \xi]_{T} > \delta_\e^4)\nonumber\\
&~\lqq  \PP(\e \|b_\e\|  e^{C_{D} T} T C_{D} > \frac{\delta_\e}{2}) + \PP([\e \xi]_{T} > \delta_\e^4)\nonumber\\
&\qquad  + \sum_{i=1}^d \sup_{x\in D_{\delta_\e}}\PP(\sup_{t\in [0, T \wedge \TT^*]} M_{t,x}^i > \frac{\delta_\e}{2d}~|~
[\e \xi]_{T} \lqq \delta_\e^4) 
+ \sup_{x\in D_{2\delta_\e}}\PP(\sup_{t\in [0, T \wedge \TT^* ]} M_{t,x}^i < -\frac{\delta_\e}{2d}~|~[\e \xi]_{T} \lqq \delta_\e^4)\nonumber\\
&~\lqq  \exp(-\delta_\e^{-1} +1) + 2d \exp\big(-\la \frac{\delta_\e}{2d} + \la^2 \Upsilon(C_{D} \la) \delta_\e^2 \big).\label{eq: P(V-u)}
\end{align}
The vanshing of the formal first term in the third to last line is the direct consequence of Lemma~\ref{lem: drift estimate}. 
We note that the last inequality is valid for any local martingale with jumps bounded from above by~$C_{D}$. 
This is satisfied since by (\ref{eq: rho-delta limit 0}) $\lim_{\e\ra 0+} \e \rho^\e = 0$ 
and for $x\in D$ and $s\in [0, \TT^*_x]$
\[
\|\Delta_s V^\e_{\cdot, x}\|\lqq \sup_{\substack{v\in \bB_{D}\\ w \in \bB_{\e \rho^\e}}}\|G(v, w)\| \lqq C_{D},
\]
where the last inequality stems from \textbf{(S.2)} part 4. 
We may now replace in inequality (\ref{eq: P(V-u)}) $\la$ 
by $2d/ \delta_\e^2$ and exploit that $\lim_{r\ra \infty}\Upsilon(r) = \frac{1}{2}$. 
This yields the desired estimate and finishes the proof. 
\end{proof}

\begin{cor}[Localization up to a fixed time $T$]\label{cor: V close to u}%\label{lem: localization in a ball 1}
For all $T>0$ there is $\e_0 \in (0,1)$ such that for all 
\mbox{$\e\in (0,\e_0]$} and $\delta_\e$ satisfying (\ref{eq: rho-delta limit}) 
follows
\begin{align*}
&~ \sup_{x\in D_{\delta_\e}} \PP(\sup_{s\in [0, T]} \|V^\e_{s,x} - u(s;x)\| > \delta_\e)\lqq \exp(-\delta_\e^{-1} + 3+ \ln(d)).
\end{align*}
\end{cor}

\begin{proof} On the event $\{\TT^*_x > T\}$ we repeat (\ref{eq: V-u}), (\ref{eq: sup V-u}) and (\ref{eq: P(V-u)}) 
replacing $t\wedge \TT^*_x$ by $t\in [0, T]$. This directly yields the desired result. 
\end{proof}

\subsection{Localization and ergodicity of $V^\e$}

\begin{lem} [Non-exit] \label{cor: non-exit}
For functions $\rho^\e$, $\delta_\e$ and $\beta_\e$ satisfing the 
relation (\ref{eq: rho-delta limit}) 
there exist constants $C>0$ and $\e_0\in (0,1)$ 
such that for all $\e \in (0,\e_0]$ 
\begin{align*}
\sup_{x\in D_{\delta_\e}} \PP(\exists\;t\in [0,T_1]: ~V^\e_{t,x} \notin D) \lqq 
\frac{ C e^{-\delta_\e^{-1}}}{(\beta_\e\delta_\e)^2}.
\end{align*}
\end{lem}

\begin{proof} 
Due to the independence of $T_1$ and $V^\e$ we calculate 
\begin{align*}
\PP(\exists\;t\in [0,T_1]: ~V^\e_{t,x} \notin D) 
& \lqq \PP(\exists\;t\in [0,\frac{1}{\beta_\e \delta_\e}]: ~V^\e_{t,x} \notin D) 
+ \PP(T_1 > \frac{1}{\beta_\e \delta_\e})  
\end{align*}
By construction 
\[
\PP(T_1 > \frac{1}{\beta_\e \delta_\e}) = e^{-\delta_\e^{-1}} \ra 0.  
\]
Recall by Remark \ref{rem D.2} that $t\gqq \sS$ and $x\in D$ imply $u(t;x) \in \iI$. Hence
\begin{align*}
\PP(\exists\;t\in [0,\frac{1}{\beta_\e \delta_\e}]: ~V^\e_{t,x} \notin D)
& = \int_0^{(\beta_\e \delta_\e)^{-1}} \beta_\e e^{-\beta_\e s} \PP(\exists\;t\in [0,s]: ~V^\e_{t,x} \notin D) ds\\
& \lqq \sum_{k=1}^{\lceil (\beta_\e \delta_\e \sS)^{-1}\rceil} \int_{(k-1)\sS}^{k \sS} 
\beta_\e e^{-\beta_\e s} \PP(\exists\;t\in [0,s]: ~V^\e_{t,x} \notin D) ds\\
& \lqq \sum_{k=1}^{\lceil (\beta_\e \delta_\e \sS)^{-1}\rceil} 
\PP(\exists\;t\in [0,k \sS]: ~V^\e_{t,x} \notin D) e^{-\beta_\e \sS k}.
% \int_{(k-1)\sS}^{k \sS} \beta_\e e^{-\beta_\e s} ds\\
\end{align*}
We denote by
\[
\eE_x(\e) := \{\sup_{t\in [0, \tT]} \|V_{t,x} -u(t;x)\| \lqq \delta_\e\}.
\]
For the case $k=1$, $x\in D_{\delta_\e}$ Corollary \ref{cor: V close to u} yields 
\[
\PP(\TT_x^* \in [0, \sS]) 
= \PP(\exists\;t\in [0,\sS]: ~V^\e_{t,x} \notin D) 
\lqq \PP(\sup_{t\in [0, \sS]} \|V^\e_{t,x} - u(t;x)\| > \delta_\e) \lqq C e^{-\delta_\e^{-1}}.
\]
Furthermore, Remark \ref{rem D.2} states that 
\[
V^\e_{\sS, x} = V^\e_{\sS, x} -u(\sS;x) + u(\sS;x) \in \bB_{\delta_\e}(0) + \iI \subset D_{2\delta_\e}.
\]
Exploiting the Markov property at time $\sS$ we obtain 
\begin{align*}
&\PP(\TT_x^* \in ((k-1)\sS, k\sS]) \\[2mm]
&= \PP(\{\forall\; t \in [0, (k-1)\sS]: ~V^\e_{t,x} \in D\} 
\cap \{\exists\;t\in [(k-1)\sS,k\sS]: ~V^\e_{t,x} \notin D\}) \\[2mm]
&= \PP(\{\forall t\in [0, (k-1)\sS]: V^\e_{t,x} \in D\} \cap \{\exists t\in [(k-1)\sS, k\sS]: ~V^\e_{t,x}\notin D\} \cap \eE_x) + \PP(\eE_x^c)\\
&\lqq \sup_{x\in D_{2\delta_\e}} \PP(\TT_x^*(\e) \in [(k-2)\sS, (k-1)\sS]) + C e^{-\delta_\e^{-1}}.
\end{align*}
Therefore a recursive argument leads to 
\begin{align*}
\PP(\TT_x^* \in ((k-1)\sS, k\sS]) \lqq k C e^{-\delta_\e^{-1}}.
\end{align*}
Finally summing up we obtain the desired result
\begin{align*}
\PP(\exists\;t\in [0,\frac{1}{\beta_\e \delta_\e}]: ~V^\e_{t,x} \notin D)
&\lqq \sum_{k=1}^{\lceil(\beta_\e\delta_\e\sS)^{-1}\rceil}  k C e^{-\delta_\e^{-1}} 
\lqq \frac{ C e^{-\delta_\e^{-1}}}{(\beta_\e\delta_\e\sS)^2}.%\lra 0, \mbox{ as } \e \searrow 0.
\end{align*}
\end{proof}

\noindent The proof further yields directly that at the time of the first large jump $T_1$ the small noise 
solution $V^\e$ is not far from $\iI$. 

\begin{cor}\label{cor: at jump time in D delta} Let the assumptions of Lemma \ref{cor: non-exit} be fulfilled. 
Then for all $\kappa >0$ there is $\e_0\in (0,1)$ such that for $\e\in (0, \e_0]$ 
\begin{equation}
\sup_{x\in D_{\delta_\e}} \PP(V^\e_{T_1,x} \in \bB_\kappa(\iI)) \lqq 
\frac{ C e^{-\delta_\e^{-1}}}{(\beta_\e\delta_\e)^2}.
\end{equation}
\end{cor}

\noindent We can now state and prove the main result of this section concerning the behavior of $X^\e_{t\in [0, T_1]}$. 

\begin{prp}[Ergodicity including the first large jump] \label{prop: ergodicity} 
Let the functions $\rho_\e, \delta_\e, \beta_\e$ satisfy (\ref{eq: rho-delta limit}) for $p=4$ 
and 
\begin{equation}\label{eq: exponentiell delta beta}
\lim_{\e\ra 0+}\frac{e^{-\delta_\e^{-1}}}{(\beta_\e\delta_\e)^2} = 0. 
\end{equation}
Consider a set $U\in \mathfrak{B}(\RR^d)$ such that 
\begin{equation}\label{eq: boundary zero}
\lim_{t\ra \infty} \sup_{x\in D} \frac{1}{t} \int_0^t \mu(E^{\partial U}(u(s;x))) ds =0. 
\end{equation}
Further, we consider a family $U^\e \in \mathfrak{B}(\RR^d)$ 
such that for all $\kappa >0$ there exists $\e_0 \in (0,1)$ satisfying for $\e\in (0, \e_0]$ 
that 
$U^\e \bigtriangleup U \subset \bB_{\kappa}(\partial U)$. 
Then 
\begin{equation}
\lim_{\e \ra 0+} \sup_{x\in D_{\delta_\e}} \big|\EE\left[e^{-T_1 \la_\e} 
\ind\{V^\e_{T_1,x} + G(V^\e_{T_1, x}, \e W) \in U^\e \}\right] 
- \int_\aA  \PP(v + G(v, \e W)\in U) P(dv) \big| = 0. 
\end{equation}
\end{prp}

\begin{proof}
Let $\theta\in (0,1)$. 
Then by Hypothesis (D.2) there is 
$\tT = \tT_\theta>0$ such that for all $t\gqq \tT$
\begin{equation}\label{eq: limit measure approximation}
\sup_{x\in D}|\frac{1}{\tT} \int_0^{\tT}  \phi(u(s;x)) ds - \int_\aA \phi(v) P(dv)|\lqq \frac{\theta}{2}.
\end{equation}
In addition, we choose $\tT>\sS$. Furthermore there exists $\kappa>0$ such that 
\[
\sup_{x\in D} \frac{1}{\tT} \int_0^\tT 
\mu(E^{\bB_{\kappa}(\partial U)}(u(s;x))) ds \lqq \frac{\theta}{2}. 
\]
Once again we lighten notation $V = V^\e$. 
Due to the independence of $T_1$ and $V$ we may continue for $x\in D_{\delta_\e}$
\begin{align*}
& \EE[e^{-T_1 \la_\e} \ind\{V_{T_1, x}+ G(V_{T_1, x}, \e W) \in U^\e\}]\\[2mm]
&\quad \lqq \EE[e^{-T_1 \la_\e} \ind\{V_{T_1, x}+ G(V_{T_1, x}, \e W) \in U^\e\}] \\
&\quad \lqq \sum_{k= 0}^{\infty} \EE\int_{k \tT}^{(k+1) \tT} \beta_\e e^{-(\beta_\e + \la_\e) s} 
\ind\{V_{s, x}+ G(V_{s, x}, \e W) \in U^\e\} ds]. 
\end{align*}
We define 
\[
\eE_x(\e) = \{\sup_{t\in [0, \tT]} \|V_{t,x} -u(t;x)\| \lqq \delta_\e\} 
\]
and calculate
\begin{align}
& \EE[\int_{k \tT}^{(k+1) \tT} \beta_\e e^{-(\beta_\e + \la_\e) s} 
\ind\{ V_{s, x}+ G(V_{s, x}, \e W) \in U^\e\} ds] \nonumber\\
&\lqq \beta_\e e^{-(\beta_\e + \la_\e) k\tT}  \EE[\int_{k \tT}^{(k+1) \tT}  
\ind\{V_{s, x}+ G(V_{s, x}, \e W) \in U^\e\} ds] \nonumber\\
& \lqq  \tT\beta_\e e^{-(\beta_\e + \la_\e) k\tT} ( \EE[\frac{1}{\tT}\int_{k \tT}^{(k+1) \tT}  
\ind\{V_{s, x}+ G(V_{s, x}, \e W) \in U^\e\} \ind(\eE_x)ds] + \sup_{x\in D_{\delta_\e}}\PP(\eE_x^c))\nonumber\\
& = \tT\beta_\e e^{-(\beta_\e + \la_\e) k\tT} 
( \EE[\EE[\frac{1}{\tT}\int_{k \tT}^{(k+1) \tT}  
\ind\{V_{s, x}+ G(V_{s, x}, \e W) \in U^\e\} \ind(\eE_x)ds~|~\fF_{\tT}]] 
+ \sup_{x\in D_{\delta_\e}}\PP(\eE_x^c))\nonumber\\
& \lqq \tT\beta_\e e^{-(\beta_\e + \la_\e) k\tT} 
(\sup_{y\in \bB_{\delta_\e}(\iI)} \EE[\frac{1}{\tT}\int_{(k-1) \tT}^{k \tT}  
\ind\{V_{s, x}+ G(V_{s, x}, \e W) \in U^\e\} ]+ \sup_{x\in D_{\delta_\e}}\PP(\eE_x^c)).\label{eq: Markov property}
\end{align}
A recursive argument yields 
\begin{align*}
& \EE[\int_{k \tT}^{(k+1) \tT} \beta_\e e^{-(\beta_\e + \la_\e) s} 
\ind\{ V_{s, x}+ G(V_{s, x}, \e W) \in U^\e\} ds] \\
&\lqq \tT\beta_\e e^{-(\beta_\e + \la_\e) k\tT}  
(\sup_{y\in D_{\delta_\e}} \frac{1}{\tT} \int_{0}^{\tT} 
\PP(V_{s, y} + G(V_{s, y} , \e W) \in U^\e) ds \\
&\qquad + (k+1)\sup_{y\in D_{\delta_\e}} \PP(\eE_y^c)) = J.
\end{align*}
We choose $\e_0\in (0,1)$ small enough such that $\e\in (0, \e_0]$ implies 
\[
(U^\e \bigtriangleup U) + \bB_{(1+Le^{L^2}) \delta_\e}(0) \subset \bB_{\kappa}(\partial U).  
\]
Hence we may continue 
\begin{align*}
J &\lqq \tT\beta_\e e^{-(\beta_\e + \la_\e) k\tT}  
(\sup_{y\in D_{\delta_\e}} (\frac{1}{\tT} \int_{0}^{\tT} 
\PP(u(s;y) + G(u(s;y) , \e W) \in U^\e + \bB_{(1+ L e^{L^2})\delta_\e}) ds \\
&\qquad + (k+1) \exp(-\delta_\e^{-1} + 3+ \ln(d))). 
\end{align*}
The first summand in the brackets satisfies due to the regular variation of $\nu$, 
the measure continuity and conditions (\ref{eq: boundary zero}) and (\ref{eq: limit measure approximation})
\begin{align*}
& \frac{\beta_\e}{h_\e} \frac{1}{\tT} \int_{0}^\tT \PP(u(s;x) + G(u(s;x), \e W) 
\in \bB_{\kappa}(U^\e)) 
\bigtriangleup U ) ds \\
& \quad \lqq   \frac{1}{\tT} \int_{0}^\tT \frac{1}{h_\e} 
\nu\Big(\frac{1}{\e}E^{\bB_{\kappa}(\partial U)}(u(s;x))\Big) ds \\
& \quad \lqq   (1+\theta) \frac{1}{\tT} \int_{0}^\tT \mu\Big(E^{\bB_{\kappa}(\partial U)}(u(s;x)\Big) ds 
\lqq (1+\theta)\frac{\theta}{2}.
\end{align*}
Hence 
\begin{align*}
& \sup_{y\in D_{\delta_\e}} (\frac{1}{\tT} \int_{0}^{\tT} 
\PP(u(s;y) + G(u(s;y) , \e W) \in \bB_{\kappa}(U^\e)) ds\\
&\quad \lqq (1+\theta) \frac{1}{\tT} \int_{0}^\tT \mu\Big(E^U(u(s;x))\Big) ds 
+ (1+\theta) \frac{\theta}{2} \frac{h_\e}{\beta_\e}
\end{align*}
We eventually obtain 
\begin{align*}
\frac{1}{\tT} \int_{0}^\tT \mu\Big(E^U(u(s;x))\Big) ds
\quad \lqq (1+\theta) \int_\aA \mu\Big(E^U(v)\big) P(dv). 
\end{align*}
Summing up over $k$ we end up with an $\e_0 \in (0,1)$ such that for $\e \in (0, \e_0]$ 
\begin{align*}
J &\lqq (1+\theta)^2 \frac{\tT~ \beta_\e}{1-e^{-\beta_\e \tT}} 
 \Big(\int_\aA \mu\Big(E^U(v)\big) P(dv)
+ (1+\theta) \frac{\theta}{2} \frac{h_\e}{\beta_\e} \Big)
+ \frac{\beta_\e}{(1-e^{-\beta_\e \tT})^2} \exp(-\delta_\e^{-1} + 3+ \ln(d))\\
&\lqq (1+\theta)^3  \left(\int_{\aA} \mu(E^U(v)) P(dv) 
+ \frac{\theta}{2}\right)
\end{align*}
This closes the proof. 
\end{proof}

\section{Proof of the Theorem \ref{thm: first exit times}} 
In this section we exploit the results on $(X^\e_{t,x})_{t\in [0, T_1]}$ and the strong Markov property 
to pass from $[0, T_1]$ to $[T_{k-1}, T_k]$ in order to determine the first exit scenario of $(X^\e_{t,x})_{t\gqq 0}$. 
The main step consists in the upper bound of the Laplace transform. 

\subsection{The upper bound}

\begin{prp}\label{prp: the upper bound}
Assume Hypotheses \textbf{(D.1)} and \textbf{(S.1-2)} to be satisfied. 
We choose $\delta_\e = \e^{\gamma}$ for $\gamma>0$ and $\rho^\e = \e^{-\rho}$ for $\rho\in (0,1)$ 
such that conditions (\ref{eq: rho-delta limit}) for $p=4$ and (\ref{eq: exponentiell delta beta}) 
are satisfied. Furthermore we assume that   
\begin{equation}\label{eq: asymptotics symmetric difference}
\int_{\aA} \mu(E^{\partial D}(y)) P(dy) = 0 \qquad \mbox{ and } \qquad \int_{\aA} \mu(E^{D^c}(y)) P(dy) >0.
\end{equation}
Then for all $\theta>0$ and $U\in \mathfrak{B}(\RR^d)$ such that 
\begin{equation}\label{eq: Rand U null}
\int_\aA \mu(E^{\partial U}(y)) P(dy) = 0 
\end{equation}
and $C\in (0,1)$ there is $\e_0\in (0,1)$ such that for $\e\in (0,\e_0]$ 
the first exit time $\TT_y = \TT_y(\e)$ satisfies 
\begin{align*}
\sup_{y\in D_{\delta_\e}} \EE\left[ e^{-\theta Q(D^c) h_\e \TT_y}  
\ind\{X^\e_{\TT_y ,y}\in U\}\right] 
\lqq (1+C) \frac{1}{1+ \theta} \frac{Q(U \cap D^c)}{Q(D^c)}. 
\end{align*}
\end{prp}
 
\begin{proof} We start by lightening the notation. 
Whenever we consider the first jump $i=1$ we omit the index. % and stress instead the dependence on $\e$. 
Hence we write $T = T_1 = T_1^\e$, $W = W_1 = W_1^\e$ etc. Define $\tau_i := T_{i}- T_{i-1}$. 
All processes will loose their $\e$ index. For convenience we abreviate $Q = Q(D^c)$. 
We define the following events for $y\in D_{\delta_\e}$ and $s, t\gqq 0$ by 
\begin{align*}
A_{t,s,y} &:= \{X_{r,\cdot}\circ \theta_{s}(y)\in D \mbox{ for all } r\in [0, t]\},\\
B_{t,s,y} &:= \{X_{r,\cdot}\circ \theta_{s}(y)\in D \mbox{ for all } r\in [0, t), 
X_{t,\cdot} \circ \theta_{s}(y) \notin D\}\\
O_{t,s,y}(U) &:= \{X_{t, \cdot}\circ \theta_{s}(y) \in U\}. % \qquad U\in \fB(\RR^d).
\end{align*}
For $x\in D_{\delta_\e}$ and with the convention $T_0 = 0$ we denote the trivial disjoint repartition 
\[
\{\TT_x <\infty\} = \bigcup_{k=1}^\infty \{\TT_x \in (T_{k-1}, T_k)\}\cup\{\TT_x = T_k\}. 
\]
Furthermore consider for $k\gqq 1$ and
\[
\{\TT_x = T_k\} = \bigcap_{i=1}^{k-1} A_{\tau_i, T_{i-1}, X_{T_{i-1}, x}} \cap B_{\tau_k, T_{k-1}, X_{T_{k-1}, x}}
\]
and analoguously
\[
\{\TT_x  \in (T_{k-1}, T_k) \} = \bigcap_{i=1}^{k-1} A_{\tau_i, T_{i-1}, X_{T_{i-1}, x}}
\cap \{V^{k}_{t}\circ \theta_{T_{k-1}}(x) \notin D \mbox{ for some }t \in (0, \tau_k)\}.
\]
Therefore we may calculate 
\begin{align*}
& \ind\{\TT_x  = T_k\} = \prod_{i=1}^{k-1} 
\ind(A_{\tau_i, T_{i-1}, X_{T_{i-1}, x}}) \ind(B_{\tau_k, T_{k-1}, X_{T_{k-1}, x}}), \\
\end{align*}
for $k=1$ 
\begin{align*}
\ind\{\TT_x  \in (0, T_1) \} &= 
\ind(\{V_{t,x} \notin D \mbox{ for some }t \in (0, T_1)\}
\end{align*}
and for $k\gqq 2$ 
\begin{align*}
 \ind\{\TT_x  \in (T_{k-1}, T_k) \} 
= \prod_{i=1}^{k-1} \ind(A_{\tau_i, T_{i-1}, X_{T_{i-1}, x}}) 
\ind(\{V^{k}_{t}\circ \theta_{T_{k-1}}(x) \notin D_{\delta_\e} \mbox{ for some }t \in (0, \tau_k)\}). 
\end{align*}
We choose $\kappa_\e := \lceil \frac{1}{h_\e}\rceil$. Hence 
\begin{align*}
& \sup_{x\in D_{\delta_\e}}\EE \left[e^{-\theta Q h_\e \TT_{x}} \ind\{X_{\TT_x , x} \in U \}\right]\\  %\ind(D_x(\TT_{x}))
& \lqq \sum_{k=1}^{\kappa_\e-1} \sup_{x\in D_{\delta_\e}}\EE\left[ e^{-\theta Q h_\e \TT_{x}} 
(\ind\{\TT_x  = T_k\}+\ind\{\TT_x  \in (T_{k-1}, T_k) \})\ind(O_{\TT_x, 0, x}(U))\right]\\[2mm]
& + \sum_{k=\kappa_\e}^{\infty} \sup_{x\in D_{\delta_\e}}\EE\left[ e^{-\theta Q h_\e \TT_{x}} 
\ind\{\TT_x  \in (T_{k-1}, T_k] \}\right]\\[2mm]
& =: S_{1} + S_{2} + S_3. %\ind(D_x(\TT_{x}))
\end{align*}
First we treat the easiest sum. 
\paragraph{1) Estimate of $S_3$: } Due to $T_k = \tau_1 + \dots + \tau_k$ 
and the independence and stationarity of $(\tau_i)$ we obtain 
\begin{align*}
S_3 &\lqq \sum_{k=\kappa_\e}^\infty \EE[e^{-\theta Q h_\e T_1}]^k
 = \sum_{k=\kappa_\e}^\infty \frac{1}{(1+ \frac{\theta Q h_\e}{\beta_\e})^k} 
 = \sum_{k=\kappa_\e}^\infty e^{k \ln(1-\frac{\theta Q h_\e}{\beta_\e})}. 
\end{align*}
There is $\e_0 \in (0,1)$ such that $\e\in (0,\e_0]$ 
\begin{align*}
S_3 &\lqq \sum_{k=\kappa_\e}^\infty e^{-k 2\frac{\theta Q h_\e}{\beta_\e}} = 
\frac{e^{-\kappa_\e 2\frac{\theta Q h_\e}{\beta_\e}}}{1-e^{-2\frac{\theta Q h_\e}{\beta_\e}}} 
\lqq \frac{2 e^{-\kappa_\e 2\frac{\theta Q h_\e}{\beta_\e}}}{2\frac{\theta Q h_\e}{\beta_\e}} \lqq \frac{C}{3}.
\end{align*}
\paragraph{2) Estimate of $S_1$:} We continue 
\begin{align*}
S_{1} & \lqq \sum_{k=1}^{\kappa_\e} \sup_{x\in D_{\delta_\e}}
\EE\left[ e^{-\theta Q h_\e T_k}\ind\{\TT_x  = T_k\}\ind(O_{T_k, 0,x}(U))\right]\\ % \ind(D_x(T_k))\ind\{\TT_x = T_k\}
& \lqq \sum_{k=1}^{\kappa_\e} \sup_{x\in D_{\delta_\e}}\EE\left[ \prod_{i=1}^{k-1}e^{-\theta Q h_\e \tau_i} 
\ind(A_{\tau_i, T_{i-1}, X_{T_{i-1}, x}}) \ind(B_{\tau_k, T_{k-1}, X_{T_{k-1}, x}})\ind(O_{\tau_{k}, T_{k-1},X_{T_{k-1}, x}}(U))\right].\\
\end{align*}
Exploiting  the same reasoning as in inequality (\ref{eq: Markov property}) 
with the strong Markov property of $X^\e$ for the jump times $(T_k)_{k\gqq 1}$ 
instead of Markov property at deterministic times $k\tT$, 
and the independence and stationarity of the increments 
we estimate the $k$-th summand of $S_{1}$ by 
\begin{align}
&\sup_{x\in D_{\delta_\e}} \EE\bigg[ \prod_{i=1}^{k-1}e^{-\theta Q h_\e \tau_i} 
\ind(A_{\tau_i, T_{i-1}, X_{T_{i-1}, x}}) 
e^{-\theta Q h_\e \tau_{k-1}} \ind(B_{\tau_k, T_{k-1}, X_{T_{k-1}, x}}) 
\ind(O_{\tau_{k}, T_{k-1},X_{T_{k-1}, x}}(U))\bigg]\nonumber\\
& \lqq \sup_{x\in D_{\delta_\e}} \EE\bigg[\bigg(e^{-\theta Q h_\e \tau_1} 
\ind(A_{\tau_1, 0, x}) \ind\{V_{T_1,x} \in D_{\delta_\e}\} 
+ \ind\{V_{T_1,x} \notin D_{\delta_\e}\}\bigg)\nonumber\\
& \qquad \EE\bigg[\prod_{i=2}^{k-1}e^{-\theta Q h_\e \tau_i} 
\ind(A_{\tau_i, T_{i-1}, X_{T_{i-1}, x}})
e^{-\theta Q h_\e \tau_{k-1}} \ind(B_{\tau_k, T_{k-1}, X_{T_{k-1}, x}})
\ind(O_{\tau_{k}, T_{k-1},X_{T_{k-1}, x}}(U))~|~\fF_{T_{1}}\bigg] \bigg] \nonumber\\
& \lqq  \sup_{x\in D_{\delta_\e}} \EE\bigg[e^{-\theta Q h_\e T} \ind(A_{x})\bigg]\nonumber \\
&\qquad \sup_{x\in D_{\delta_\e}} \EE\bigg[ \prod_{i=1}^{k-2}e^{-\theta Q h_\e \tau_i} 
\ind(A_{\tau_i, T_{i-1}, X_{T_{i-1}, x}}) 
e^{-\theta Q h_\e \tau_{k-2}} 
\ind(B_{\tau_{k-1}, T_{k-2}, X_{T_{k-2}, x}})
\ind(O_{\tau_{k-1}, T_{k-2},X_{T_{k-2}, x}}(U)) \bigg]\nonumber\\
&\qquad +  \sup_{x\in D_{\delta_\e}} \PP(V_{T_1,x}\notin D_{\delta_\e}) \label{eq: Strong Markov estimate}
\end{align}
where we use the abreviation
\begin{align*}
A_x &= A_{T_1, 0, x} \\
B_x &= B_{T_1, 0, x} \\
O_x^{U} &= O_{T_1, 0, x}(U). 
\end{align*}
The recursion from $k-1$ to $1$ leads to 
\begin{align}
&\sup_{x\in D_{\delta_\e}} \EE\bigg[ \prod_{i=1}^{k-1}e^{-\theta Q h_\e \tau_i} 
\ind(A_{\tau_i, T_{i-1}, X_{T_{i-1}, x}}) 
e^{-\theta Q h_\e \tau_{k-1}} 
\ind(B_{\tau_k, T_{k-1}, X_{T_{k-1}, x}}) 
\ind(O_{\tau_{k}, T_{k-1},X_{T_{k-1}, x}}(U))\nonumber \\ 
& \lqq   \bigg(\sup_{x\in D_{\delta_\e}} \EE\bigg[e^{-\theta Q h_\e T} \ind(A_{x})\bigg]\bigg)^{k-1} 
\sup_{x\in D_{\delta_\e}} \EE\bigg[e^{-\theta Q h_\e T} \ind(B_{x})
\ind(O^U_{x})) \bigg]\nonumber\\
&\qquad + \sup_{x\in D_{\delta_\e}} \PP(V_{T_1,x}\notin D_{\delta_\e})
\sum_{j=0}^{k-2} \bigg(\sup_{x\in D_{\delta_\e}} \EE\bigg[e^{-\theta Q h_\e T} \ind(A_{x})\bigg]\bigg)^{j}.
\label{eq: Strong Markov estimate recursion}
\end{align}
In the same way we estimate the $k$-th summand of $S_2$ for $k\gqq 1$. 
\begin{align*}
&\sup_{x\in D_{\delta_\e}} \EE\bigg[ \prod_{i=1}^{k-1}e^{-\theta Q h_\e \tau_i} 
\ind(A_{\tau_i, T_{i-1}, X_{T_{i-1}, x}}) 
e^{-\theta Q h_\e \tau_{k-1}} 
\ind(\{V^{k}_{t}\circ \theta_{T_{k-1}}(x) \in D_{\delta_\e}^c \cap U \mbox{ for some }t \in (0, \tau_k)\})]\\
& \lqq   \bigg(\sup_{x\in D_{\delta_\e}} \EE\bigg[e^{-\theta Q h_\e T} \ind(A_{x})\bigg]\bigg)^{k-1}
\sup_{x\in D_{\delta_\e}} \EE\bigg[e^{-\theta Q h_\e T} \ind(\{V^{\e}_{t, x} \in D \cap U 
\mbox{ for some }t \in (0, T_1)\}) \bigg]\\
&\qquad + \sup_{x\in D_{\delta_\e}} \PP(V_{T_1,x}\notin D_{\delta_\e})
\sum_{j=0}^{k-2} \bigg(\sup_{x\in D_{\delta_\e}} \EE\bigg[e^{-\theta Q h_\e T} \ind(A_{x})\bigg]\bigg)^{j}. \nonumber
 \end{align*}
We show now that the first sum can be estimated by $1/(1+\theta)$, the Laplace transform of $\mbox{EXP}(1)$ evaluated at $\theta$, 
plus a small error and that both additional sums tend to zero if $\e$ does so.

\noindent \textbf{Starting with the first factor of the main sum} we obtain 
\begin{align*}
\sup_{y\in D_{\delta_\e}} \EE\left[ e^{-\theta Q h_\e T}  \bI(A_y) \right] 
&\lqq \sup_{y\in D_{\delta_\e}} ~\EE\left[e^{-\theta Q h_\e T} (1-\bI\{V_{T, y} + G(V_{T, y}, \e W) 
\in D^c \})\right].
\end{align*}
Proposition \ref{prop: ergodicity} and the independence of $W$ from $T$ and $V$ %Lemma \ref{lem: ergodic property in detail} 
ensure the existence $\e_0\in (0,1)$ such that for $\e \in (0, \e_0]$ 
\begin{align*}
\sup_{y\in D_{\delta_\e}} \EE\left[ e^{-\theta Q h_\e T}  \bI(A_y) \right]  &  \lqq \frac{\beta_\e}{\theta Q h_\e + \beta_\e} 
\left(1 - (1-C) \int_\aA  \PP(v + G(v, \e W)\in D^c) P(dv)\right).
\end{align*}
\noindent Since by definition 
\begin{align*}
\PP\Big(v + G(v, \e W)\in D^c\Big) 
= \frac{1}{\beta_\e} \nu\Big(\frac{1}{\e} E^{D^c}(v))\Big),
\end{align*}
$\aA$ is compact and the distance $d(\aA, \pd D) >0$, 
the regular variation of $\nu$ implies the existence of a constant $\e_0\in (0,1)$ 
such that 
\begin{align*}\displaystyle
\sup_{v\in \aA} \bigg|\frac{\PP\Big(v + G( v, \e W)\in D^c\Big)}
{\displaystyle\frac{h_\e}{\beta_\e} \mu\left(E^{D^c}(v)\right) } -1\bigg| \lqq C, 
\quad \mbox{ for all }\e\in (0, \e_0].
\end{align*}
Hence there is $\e_0 \in (0,1)$ such for $\e\in (0, \e_0]$ 
\begin{align}
\sup_{y\in D_{\delta_\e}} \EE\left[ e^{-\theta Q h_\e T}  \bI(A_y) \right]  &~\lqq \frac{\beta_\e}{ \theta Q h_\e + \beta_\e} \left(1 - (1-C)^2\frac{h_\e}{\beta_\e} 
\int_{\aA} \mu\left(E^{D^c}(u)\right) P(du)\right) \nonumber\\
&~=  \frac{\beta_\e}{ \theta Q h_\e + \beta_\e} 
\left(1 - (1-C)^2\frac{Q h_\e}{\beta_\e}\right). \label{eq: estimate A}
\end{align}

\noindent \textbf{The second factor of the main sum} can be treated analogously 
and we obtain for sufficiently small $\e_0\in (0,1)$ that 
\begin{equation}\label{eq: estimate B}
\sup_{y\in D_{\delta_\e}} \EE\left[ e^{-\theta Q h_\e T} \bI(B_y)\ind(O^{U}_y) \right] 
\lqq (1+C)^2 \frac{\beta_\e}{\theta Q h_\e + \beta_\e} \frac{Q(U\cap D^c) h_\e}{\beta_\e} 
\end{equation}
for $\e \in (0, \e_0]$, where 
\[
Q(D^c \cap U) =  \int_\aA \mu\left(E^{D^c \cap U}(u)\right) P(du). 
\]

\paragraph{For the remainder sum} we exploit Corollary \ref{cor: at jump time in D delta} 
which yields a constant $C'>0$ and $\e_0\in (0,1)$ such that for $\e \in (0,\e_0]$ 
\begin{align*}
\sup_{y\in D_{\delta_\e}}\PP(\TT_y  \in (0, T)) \lqq \frac{ C' e^{-\delta_\e^{-1}}}{(\beta_\e\delta_\e)^2}
\end{align*}
and obtain 
\begin{align*}
\sup_{x\in D_{\delta_\e}} \PP(V_{T_1,x}\notin D_{\delta_\e})& \sum_{k=1}^{k_\e-1} 
\sum_{j=0}^{k-2} \bigg(\sup_{x\in D_{\delta_\e}} \EE\bigg[e^{-\theta Q h_\e T} \ind(A_{x})\bigg]\bigg)^{j} \\
& \lqq \frac{ C' e^{-\delta_\e^{-1}}}{(\beta_\e\delta_\e)^2} 
\sum_{k=1}^{\kappa_\e-1} \sum_{j=0}^{k-2}  \left(\frac{\beta_\e}{ \theta Q h_\e + \beta_\e} 
\left(1 - (1-C)^2\frac{Q h_\e}{\beta_\e}\right)\right)^j.
\end{align*}
Let us call 
\[
q_\e = \frac{\beta_\e}{ \theta Q h_\e + \beta_\e} 
\left(1 - (1-C)^2\frac{Q h_\e}{\beta_\e}\right).
\]
Then 
\begin{align}\label{eq: remainder sum}
\sup_{x\in D_{\delta_\e}} \PP(V_{T_1,x}\notin D_{\delta_\e})& \sum_{k=1}^{k_\e-1} 
\sum_{j=0}^{k-2} \bigg(\sup_{x\in D_{\delta_\e}} \EE\bigg[e^{-\theta Q h_\e T} \ind(A_{x})\bigg]\bigg)^{j} \nonumber\\
& =: \frac{ C' e^{-\delta_\e^{-1}}}{(\beta_\e\delta_\e)^2} \sum_{k=0}^{\kappa_\e-2} \frac{1-q_\e^{k-1}}{1-q_\e} \nonumber\\
&\lqq \frac{ C' e^{-\delta_\e^{-1}}}{(\beta_\e\delta_\e)^2} \frac{\kappa_\e}{1-q_\e} \lqq  \frac{C}{3}. 
\end{align}
Eventually inequalities (\ref{eq: estimate A}), (\ref{eq: estimate B}) and (\ref{eq: remainder sum}) combined 
imply the existence of $\e_0\in (0,1)$ such that 
\[
S_1 \lqq (1+C)^2 \frac{\beta_\e}{\theta Q h_\e + \beta_\e} 
\frac{Q(D^c \cap U) h_\e}{\beta_\e} \sum_{k=1}^\infty 
\left( \frac{\beta_\e}{\theta Q h_\e + \beta_\e} \left(1-\frac{Qh_\e}{\beta_\e}\left(1 - C\right)^2\right)\right)^{k-1} + \frac{C}{3}
\]
for $\e \in (0,\e_0]$.

\paragraph{3) Estimate of $S_2$:} For $k=1$ we exploit Corollary \ref{cor: non-exit}, 
which yields a constant $C'>0$ and $\e_0\in (0,1)$ such that for $\e \in (0,\e_0]$ 
\begin{align*}
\sup_{y\in D_{\delta_\e}}\PP(\{\TT_y  \in (0, T)\} \cap O_{\TT_y, 0, x}(U)) 
&\lqq \sup_{x\in D_{\delta_\e}} \PP(t\in [0,T]: ~V_{t,x} \notin D_{\delta_\e} \cap U)\\
&\lqq \frac{ C' e^{-\delta_\e^{-1}}}{(\beta_\e\delta_\e)^2}\\
&\lqq ((1+C)^3-(1+C)^2) \frac{\beta_\e}{\theta Q h_\e + \beta_\e} \frac{Qh_\e}{\beta_\e}.
\end{align*}
The last estimate follows by the algebraic choice of $\delta_\e$. 
and eventually with the help of estimate (\ref{eq: remainder sum}) of the remainder sum
\begin{align*}
S_2 & \lqq ((1+C)^3-(1+C)^2) \frac{\beta_\e}{\theta Q h_\e + \beta_\e} \frac{Q(D^c \cap U) h_\e}{\beta_\e} 
\sum_{k=1}^\infty\left( \frac{\beta_\e}{ \theta Q h_\e + \beta_\e}
 \left(1 - (1-C)^2\frac{Q h_\e}{\beta_\e}\right)\right)^{k-1} + \frac{C}{3}\\
\end{align*}

\paragraph{Conclusion:} We infer that there is a sufficiently small constant $\e_0\in (0,1)$ such that for $\e \in (0, \e_0]$ 

\begin{align*}
\sup_{x\in D_{\delta_\e}} \E&\left[ e^{-\theta Q h_\e \TT_x} \ind\{X^\e_{\TT_y ,y}\in U\}\right] \lqq S_1 + S_2 + S_3\\
&\lqq ~(1+C)^3 \frac{\beta_\e}{\theta Q h_\e + \beta_\e} \frac{Q(D^c \cap U)h_\e}{\beta_\e} 
\sum_{k=1}^\infty \left(\frac{\beta_\e}{\theta Q h_\e + \beta_\e} \left(1-\frac{Qh_\e}{\beta_\e}\left(1 - C \right)^3\right)\right)^{k-1} \\[2mm]
&= ~\frac{\ds (1+C)^3\frac{\beta_\e}{\ds \theta Q h_\e + \beta_\e} 
\frac{Q(D^c \cap U)h_\e}{\beta_\e}}{\ds 1- \left(\frac{\beta_\e}{\theta Q h_\e + \beta_\e} 
\left(1-\frac{Qh_\e}{\beta_\e}\left(1 - C \right)^3\right)\right)} + C\\[3mm]
& = \frac{(1+C)^3}{\theta + (1 - C)^3}\frac{Q(D^c \cap U)}{Q} + C.
\end{align*}
By an appropriate renaming of the constant $C$ we close the proof. 
\end{proof}

\subsection{The lower bound}

\begin{prp}\label{prp: the lower bound}
Let the assumptions of Proposition (\ref{prp: the upper bound}) be satisfied.
Then for all $\theta>0$, $U\in \mathfrak{B}(\RR^d)$ satisfying (\ref{eq: Rand U null})
and $C\in (0,1)$ there is $\e_0\in (0,1)$ such that for $\e\in (0,\e_0]$ 
the first exit time $\TT_y = \TT_y(\e)$ satisfies 
\begin{align*}
\inf_{y\in D_{\delta_\e}} \EE\left[ e^{-\theta Q(D^c) h_\e \TT_y} \ind\{X^\e_{\TT_y,y}\in U\}\right] 
\gqq \frac{Q(D^c \cap U)}{Q(D^c)} \frac{1-C}{1+ \theta +C}. 
\end{align*}
\end{prp}

\begin{proof} We keep the notation introduced in the proof of Proposition \ref{prp: the upper bound}.  
We define the following events for $y\in D_{\delta_\e}$ and $t, s\gqq 0$ by 
\begin{align*}
A^-_{t,s,y} &= \{X_{r,\cdot}\circ \theta_{s}(y)\in D \mbox{ for all } r\in [0, t) 
\mbox{ and }X_{t,\cdot}\circ \theta_{s}(y)\in D_{\delta_\e}\}
\end{align*}
and the abbreviation 
\[
A^-_{y} = A^-_{T_1, 0, y}.
\]
The identical strong Markov property estimates from below (\ref{eq: Strong Markov estimate}) and (\ref{eq: Strong Markov estimate recursion}) 
as in the proof of Proposition \ref{prp: the upper bound} only with inverted inequalities and neglecting all the nonnegative error terms 
yields 
\begin{align*}
& \inf_{y\in D_{\delta_\e}} \EE\left[ e^{-\theta Q h_\e \TT_y}  \ind\{X^\e_{\TT_y ,y}\in U\}\right]\\
& ~\gqq ~\sum_{k=1}^\infty \left(\inf_{y\in D_{\delta_\e}} \EE\left[ e^{-\theta Q h_\e T} \bI(A_y^{-}) \right]\right)^{k-1} 
\inf_{y\in D_{\delta_\e}} \EE\left[ e^{-\theta Q h_\e T} \bI(B_y) \ind(D^{U}_y)\right] .
\end{align*} 
Proposition \ref{prop: ergodicity} yields a constant $\e_0\in (0,1)$ such that for $\e\in (0, \e_0]$ 
\begin{align*}
\inf_{y\in D_{\delta_\e}}\EE\left[ e^{-\theta Q h_\e T} \bI(A_y^-) \right]
& \gqq  \frac{\beta_\e}{\theta Q h_\e + \beta_\e} (1- (1+C)^2\frac{Q h_\e}{\beta_\e} (D^c))\\
\end{align*}
and 
\begin{align*}
\inf_{y\in D_{\delta_\e}}\EE\left[ e^{-\theta Q h_\e T} \bI(B_y) \bI(D^{U}_y)\right]
% & \gqq (1-\frac{C}{4}) \frac{\beta_\e}{\theta Q h_\e + \beta_\e} \int_\aA \PP(u + G(u, \e W) \in D^c \cap U) P(du) \\
& \gqq (1-C)^2 \frac{\beta_\e}{\theta Q h_\e + \beta_\e} \frac{h_\e}{\beta_\e} Q(D^c \cap U).
\end{align*}
This eventually implies a constant $\e_0 \in (0,1)$ such that for $\e\in (0, \e_0]$ follows 
\begin{align*}
& \inf_{x\in D_{\delta_\e}}\EE\left[ e^{-\theta Q h_\e \TT_{x}}\ind\{X_{\TT_{x} ,x}\in U\}\right] \\
&\qquad \gqq (1-C)^2 \frac{\beta_\e}{\theta Q h_\e + \beta_\e}\frac{h_\e}{\beta_\e} Q(U\cap D^c) 
~\sum_{k=1}^\infty \left(\frac{\beta_\e}{\theta Q h_\e + \beta_\e} (1- (1+C)^2\frac{h_\e}{\beta_\e} Q)\right)^{k-1} \\[2mm]
&\qquad = \frac{Q(U \cap D^c)}{Q(D^c)} \frac{(1-C)^2}{\theta + (1 +C)^2}.
\end{align*}
An appropriate renaming of the constant $C$ finishes the proof. 
\end{proof}

\bigskip

\begin{proof} (Theorem \ref{thm: first exit times}) For $\rho \in (0,\frac{1}{2})$ we define $\rho^\e = \e^{-\rho}$ and 
verify condition (\ref{eq: rho-delta limit}) for $p=4$ and (\ref{eq: exponentiell delta beta}) 
for the choice of $\rho^\e$ and $\delta_\e$.  
\[
\frac{\e \rho^\e}{\delta_\e^{(p+1)/2}} = \e^{1-\rho - 5/2} \ra 0, \mbox{ as }\e \ra 0+. 
\]
Since for small $\e$ the intensity $\beta_\e \approx_\e \e^{\alpha \rho} \ell(\frac{1}{\e^\rho}) \mu(\bB_1^c(0))$ 
is asymptotically dominated by a polynomial order just as $\delta_\e$, the reasoning is reduced to the fact that 
the exponential convergence of $e^{-\delta_\e}$ dominates $(\delta_\e \beta_\e)^{-2}$ in the limit as $\e\ra 0+$. 
This implies relation (\ref{eq: exponentiell delta beta}). 
Therefore the upper bound by Proposition \ref{prp: the upper bound} and the lower bound 
by Proposition \ref{prp: the lower bound} are satisfied, which yields the desired result. 
\end{proof}

% \bibliographystyle{unsrt}
% \bibliographystyle{plain}
%  \bibliography{biblio-new}

\begin{thebibliography}{10}

\bibitem{Applebaum-09}
D.~Applebaum.
\newblock {\em {{L{\'e}}vy processes and stochastic calculus}}, volume 116 of
  {\em {Cambridge Studies in Advanced Mathematics}}.
\newblock Cambridge University Press, second edition, 2009.

\bibitem{BerglundG-04}
N.~Berglund and Gentz B.
\newblock {On the noise-induced passage through an unstable periodic orbit {I}:
  {T}wo-level model}.
\newblock {\em Journal of Statistical Physics}, 114(5--6):1577--1618, 2004.

\bibitem{BinghamGT-87}
N.~H. Bingham, C.~M. Goldie, and J.~L. Teugels.
\newblock {\em {Regular variation}}, volume~27 of {\em {Encyclopedia of
  Mathematics and its applications}}.
\newblock Cambridge University Press, 1987.

\bibitem{BovierEGK-04}
A.~Bovier, M.~Eckhoff, V.~Gayrard, and M.~Klein.
\newblock {Metastability in reversible diffusion processes I: Sharp asymptotics
  for capacities and exit times}.
\newblock {\em Journal of the European Mathematical Society}, 6(4):399--424,
  2004.

\bibitem{Br91}
S.~Brassesco.
\newblock {Some results on small random perturbations of an infinite
  dimensional dynamical system}.
\newblock {\em Stochastic Processes and their Applications}, 38:33--53, 1991.

\bibitem{Br96}
S.~Brassesco.
\newblock {Unpredictabililty of an exit time}.
\newblock {\em Stochastic Processes and their Applications}, 63:55--65, 1996.

\bibitem{Day-83}
M.~V. Day.
\newblock {On the exponential exit law in the small parameter exit problem}.
\newblock {\em Stochastics}, 8:297--323, 1983.

\bibitem{Da96}
M.~V. Day.
\newblock {Exit cycling for the {V}an der {P}ol oscillator and quasipotential
  calculations}.
\newblock {\em Journal of Dynamics and Differential Equations}, 8(4):573--601,
  1996.

\bibitem{DHI13}
A.~Debussche, M.~H{\"o}gele, and P.~Imkeller.
\newblock {\em {Metastability of reaction diffusion equations with small
  regularly varying noise}}.
\newblock {Lecture Notes in Mathematics}. Springer--Verlag, 2013.
\newblock To appear.

\bibitem{EFSV85}
L.~N. Epele, H.~Fanchiotti, A.~Spina, and H.~Vucetich.
\newblock {Noise-driven self-excited oscillators: {D}iffusion between limit
  cycles}.
\newblock {\em Physical Review A}, 31(4):2631--2638, 1985.

\bibitem{FL82}
G.~W. Faris and G.~Jona-Lasinio.
\newblock {Large fluctuations for a nonlinear heat equation with noise}.
\newblock {\em Journal of Physics A: Mathematical and General}, 15(10):3025,
  1982.

\bibitem{Fr88}
M.~I. Freidlin.
\newblock {Random perturbations of reaction-diffusion equations: the
  quasideterministic approximation}.
\newblock {\em Transactions of the American Mathematical Society},
  305(2):665--697, 1988.

\bibitem{FreidlinW-98}
M.~I. Freidlin and A.~D. Wentzell.
\newblock {\em {Random perturbations of dynamical systems}}, volume 260 of {\em
  {Grundlehren der Mathematischen Wissenschaften}}.
\newblock {Springer}, second edition, 1998.

\bibitem{Godovanchuk-82}
V.~V. Godovanchuk.
\newblock {Asymptotic probabilities of large deviations due to large jumps of a
  Markov process}.
\newblock {\em Theory of Probability and its Applications}, 26:314--327, 1982.

\bibitem{GM88}
A.~Goldbeter and F.~Moran.
\newblock {Dynamics of a biochemical system with multiple oscillatory domainsas
  a clue for multiple modes of neuronal oscillations}.
\newblock {\em European Biophysics Journal}, 15:277--287, 1988.

\bibitem{HLP09}
J.~M. Hill, N.~G. Lloyd, and J.~M. Pearson.
\newblock {Limit cycles of a predator--prey model with intratrophic predation}.
\newblock {\em Journal of mathematical analysis and applications},
  349(2):544--555, 2009.

\bibitem{HultL-06-1}
H.~Hult and F.~Lindskog.
\newblock {Regular variation for measures on metric spaces}.
\newblock {\em Publications de l'Institut Math{\'e}matique (Beograd). Nouvelle
  S{\'e}rie}, 80(94):121--140, 2006.

\bibitem{ImkellerP-06}
P.~{Imkeller} and I.~{Pavlyukevich}.
\newblock {First exit times of SDEs driven by stable {L}{\'e}vy processes}.
\newblock {\em Stochastic Processes and their Applications}, 116(4):611--642,
  2006.

\bibitem{ImkPavSta-10}
P.~Imkeller, I.~Pavlyukevich, and M.~Stauch.
\newblock {First exit times of non-linear dynamical systems in {$\mathbb{R}^d$}
  perturbed by multifractal {L{\'e}}vy noise}.
\newblock {\em Journal of Statistical Physics}, 141(1):94--119, 2010.

\bibitem{Kallenberg-02}
O.~Kallenberg.
\newblock {\em {Foundations of modern probability}}.
\newblock {Probability and Its Applications}. Springer, second edition, 2002.

\bibitem{Kramers-40}
H.~A. Kramers.
\newblock {Brownian motion in a field of force and the diffusion model of
  chemical reactions}.
\newblock {\em Physica}, 7:284--304, 1940.

\bibitem{KurSch91}
C.~Kurrer and K.~Schulten.
\newblock {Effect of noise and perturbations on limit cycle systems}.
\newblock {\em Physica D}, 50:311--320, 1991.

\bibitem{LC98}
J.~Liu and J.~W. Crawford.
\newblock {Stability of an autocatalytic biochemical system in the presence of
  noise perturbations}.
\newblock {\em IMA Journal of Mathematics Applied in Medicine and Biology},
  15(4):339--350, 1998.

\bibitem{Pavlyukevich11}
I.~Pavlyukevich.
\newblock {First exit times of solutions of stochastic differential equations
  driven by multiplicative {L{\'e}}vy noise with heavy tails}.
\newblock {\em Stochastics and Dynamics}, 11(2\&3), 2011.

\bibitem{Protter-04}
P.~E. Protter.
\newblock {\em {Stochastic integration and differential equations}}, volume~21
  of {\em {Applications of Mathematics}}.
\newblock {Springer}, second edition, 2004.

\bibitem{Resnick-04}
S.~Resnick.
\newblock {On the foundations of multivariate heavy-tail analysis}.
\newblock {\em Journal of Applied Probability}, 41A:191--212, 2004.

\bibitem{SV87}
Y.~A. Saet and G.~Viviani.
\newblock {The stochastic process of transitions between limit cycles for a
  special class of self-oscillators under random perturbations}.
\newblock {\em IEEE Transactions on Circuits and Systems}, CAS-34(6):691--695,
  1987.

\bibitem{Sato-99}
K.~Sato.
\newblock {\em {L{\'e}vy processes and infinitely divisible distributions}},
  volume~68 of {\em {Cambridge Studies in Advanced Mathematics}}.
\newblock {Cambridge University Press}, 1999.

\bibitem{Temam97}
R.~Temam.
\newblock {\em {Infinite-dimensional dynamical systems in Mechanics and
  Physics}}, volume~68 of {\em {Applied Mathematical Sciences}}.
\newblock Springer-Verlag, second edition, 1997.

\bibitem{FW70}
A.~D. Ventsel' and M.~I. Freidlin.
\newblock {On small random perturbations of dynamical systems}.
\newblock {\em Russian Mathematical Surveys}, 25(1):1--55, 1970.

\bibitem{Wentzell-90}
A.~D. Wentzell.
\newblock {\em {Limit theorems on large deviations for Markov stochastic
  processes}}, volume~38 of {\em {Mathematics and Its Applications (Soviet
  Series)}}.
\newblock {Kluwer Academic Publishers}, 1990.

\end{thebibliography}
% \newpage
% \section*{Acknowledgements} 
% The authors acknowledge 1

\end{document}